\documentclass[a4wide,11pt]{amsart}
\usepackage{amsmath, amsthm, amscd, amssymb}
\usepackage{tikz,cite}
\usepackage{caption,subcaption}
\usepackage{adjustbox,multirow,bm,array,makecell,diagbox}
\usepackage{wasysym}
\usepackage{tabularx, booktabs,ragged2e}
\usepackage{hyperref}
\usepackage{relsize}
\numberwithin{equation}{section}
\theoremstyle{plain}
\newtheorem{theorem}{Theorem}[section]
\newtheorem{proposition}[theorem]{Proposition}
\newtheorem{lemma}[theorem]{Lemma}
\newtheorem{corollary}[theorem]{Corollary}

\theoremstyle{definition}

\newtheorem{remark}[theorem]{Remark}
\newtheorem{example}[theorem]{Example}

\def\C{\mathbb{C}}
\def\R{\mathbb{R}}
\def\Z{\mathbb{Z}}
\def\Q{\mathbb{Q}}

\def\cF{\mathcal{F}}

\def\St{\mathrm{St}\,}
\def\Lk{\mathrm{Lk}\,}

\tikzstyle{v}=[circle, draw, solid, fill=black, inner sep=0pt, minimum width=4pt]

\DeclareMathOperator{\Cat}{Cat}
\DeclareMathOperator{\sgn}{sgn}

\newcommand{\CIII}{%
\begin{tikzpicture}%
\draw  (0,0) -- (0,1.8ex) -- (1.8ex,1.8ex) -- (1.8ex,0)-- (0,0);
\filldraw[fill=gray!50] (1.8ex,1.8ex) --(1.1ex,1.8ex) -- (1.8ex,1.1ex)--cycle;
\filldraw[fill=gray!50] (0,0) --(0.7ex,0) -- (0,0.7ex)--cycle;
\filldraw[fill=gray!50] (0,1.8ex) --(0.7ex,1.8ex) -- (0,1.1ex)--cycle;
\filldraw[fill=gray!50] (1.8ex,0) --(1.1ex,0) -- (1.8ex,0.7ex)--cycle;
\filldraw[fill=white] (1.9ex,0)--(2ex,0);
\end{tikzpicture}%
}

\begin{document}
\title{Graph invariants and Betti numbers of real toric manifolds}

\author[Boram Park]{Boram Park}
\address[B. Park]{Department of mathematics, Ajou University,
Suwon 16499, Republic of Korea}
\email{borampark@ajou.ac.kr}

\author[Hanchul Park]{Hanchul Park}
\address[H. Park]{Seoul Science High School, 63 Hyehwa-ro, Jongno-gu, Seoul 03066, Republic of Korea}
\email{hanchulp@sen.go.kr}

\author[Seonjeong Park]{Seonjeong Park}
\address[S. Park]{Applied Algebra and Optimization Research Center, Sungkyunkwan University, Suwon 16419, Republic of Korea}
\email{seonjeong1124@gmail.com}



\thanks{The first named author was supported by Basic Science Research Program through the National Research Foundation of Korea (NRF) funded by the Ministry of Science, ICT and Future Planning (NRF-2018R1C1B6003577)}

\date{\today}
\maketitle

\begin{abstract}
    For a graph $G$, a graph cubeahedron $\square_G$ and a graph associahedron $\triangle_G$ are simple convex polytopes which admit (real) toric manifolds. In this paper, we introduce a graph invariant, called the $b$-number, and we show that the $b$-numbers compute the Betti numbers of the real toric manifold $X^\R(\square_G)$ corresponding to a graph cubeahedron. The $b$-number is a counterpart of the notion of $a$-number, introduced by S.~Choi and the second named author, which computes the Betti numbers of the real toric manifold $X^\R(\triangle_G)$ corresponding to a graph associahedron. We also study various relationships between $a$-numbers and $b$-numbers from a toric topological view. Interestingly, for a forest $G$ and its line graph $L(G)$, the real toric manifolds $X^\R(\triangle_G)$ and $X^\R(\square_{L(G)})$ have the same Betti numbers.
\end{abstract}

\section{Introduction}\label{sec1}
    Throughout this paper, we focus on simple convex polytopes constructed from a graph, and we only consider a finite simple graph and  use $G$ or $H$ for a generic symbol  to denote a graph.

    For a graph $G$, the \emph{graph associahedron} of~$G$, denoted by $\triangle_G$, is a simple convex polytope obtained from a product of simplices by truncating the faces corresponding to proper connected induced subgraphs of each component of $G$; see Section~\ref{sec:simple_graph} for the precise construction. This polytope was first introduced by Carr and Devadoss in~\cite{CD2006} whose motivation was the work of De Concini and Procesi, wonderful compactifications of hyperplane arrangements~\cite{DeC-Pro1995}. Graph associahedra have also appeared in a broad range of subjects such as the moduli space of curves~\cite{Fe-Jen-Ran2016,Devadoss1999} and enumerative properties like $h$-vectors~\cite{PRW}.

    The $h$-vector of a simple polytope is a fundamental invariant of the polytope which encodes the number of faces of different dimensions, and it has been known that the $h$-vectors of graph associahedra give interesting integer sequences. For example, the $h$-vector of the graph associahedron $\triangle_{P_n}$ of a path $P_n$ is given by the Narayana numbers: $h_i(\triangle_{P_n})=N(n,i+1)=\frac{1}{n}{n\choose i+1}{n\choose i}$ for $i=0,\ldots,n-1$, see~\cite{PRW} for more {examples.}

    A \emph{graph cubeahedron}  is a simple convex polytope introduced in~\cite{DHV}, and it is deeply related to the moduli space of a bordered Riemann surface. The \emph{graph cubeahedron} of $G$, denoted by $\square_G$, is defined to be a polytope obtained from a cube by truncating the faces corresponding to  connected induced subgraphs. It was also shown in~\cite{DHV} that the graph cubeahedron $\square_{P_n}$ is combinatorially equivalent to the graph associahedron $\triangle_{P_{n+1}}$, and hence $h_i(\square_{P_n})$ is the Narayana number $N(n+1,i+1)$ for $i=0,\ldots,n$.
    But graph cubeahedra are not much known compared with graph associahedra.

    On the other hand, there is a beautiful connection between the $h$-vector of a simple polytope and the Betti numbers of a toric variety in toric geometry. {A compact non-singular toric variety (\emph{toric manifold} for short) is over a simple polytope $P$ if its quotient by the compact torus is homeomorphic to $P$ as a manifold with corners. If a toric manifold $X$ is over $P$, then} the cohomology groups of~$X$ vanish in odd degrees and the $2i$th Betti number of~$X$ is equal to $h_i(P)$, see~\cite{Jur80,Dan78}. {In fact, both graph associahedra and graph cubeahedra can admit toric manifolds over the polytopes, and hence the Betti numbers of toric manifolds associated with a path are Narayana numbers.}

    Unlike (complex) toric varieties, the real locus of a toric manifold, called a \emph{real toric manifold}, is much less known for its cohomology. In coefficient~$\Z_2$, the cohomology of a real toric manifold is  very similar to the complex case according to~\cite{Jur85}; for a toric manifold $X$ and its real locus $X^\R$, the $i$th $\Z_2$-Betti number of $X^\R$ is equal to the $2i$th Betti number of $X$, and hence it is also determined by the $h$-vector of $X/T$.\footnote{Given a topological space $X$, the $i$th \emph{Betti number} of $X$, denoted by $\beta^i(X)$, is the free rank of the singular cohomology group $H^i(X;\Z)$ and for a field $F$ the $i$th $F$-\emph{Betti number} of $X$, denoted by $\beta^i_F(X)$, is the dimension of $H^i(X;F)$ as a vector space over~$F$. Note that $\beta^i(X)=\beta^i_\Q(X)$.} But the Betti numbers of $X^\R$ are different from the $h$-vector of $X/T$ in general.
    For example, both $\C P^1\times \C P^1$ and $\C P^2\#\overline{\C P^2}$ are toric manifolds over the 2-cube $\square^2$ and have the same Betti numbers $\beta^0=1, \beta^2=2, \beta^4=1$ and $\beta^{odd}=0$. The real loci of $\C P^1\times \C P^1$ and $\C P^2\#\overline{\C P^2}$ are the $2$-dimensional torus $\mathbb{T}$ and the Klein bottle $\mathbb{K}$, respectively. Both $\mathbb{T}$ and $\mathbb{K}$ have the same $\Z_2$-Betti numbers: $\beta_{\Z_2}^0=1$, $\beta_{\Z_2}^1=2$, and $\beta_{\Z_2}^2=1$. But their Betti numbers are different; $\beta^1(\mathbb{T})=2$ and $\beta^1(\mathbb{K})=1$.

    Recently, the rational cohomology groups of real toric manifolds were studied in \cite{CP2,CP3,ST2012,Tre2012}, and the results make us look the rational cohomology groups of real toric manifolds combinatorially. However, it is {still} difficult to {compute} their Betti numbers {explicitly} by the results in general. For some interesting families of real toric manifolds, their rational cohomology have been studied in~\cite{CKP,CP,CPP2015,PP2017, CPP_B}.
    In this paper, we study the integral cohomology groups of real toric manifolds arising from a graph, in a combinatorial way.

    A simple convex polytope of dimension $n$ is called a \emph{Delzant polytope} if the $n$~primitive integral vectors (outward) normal to the facets meeting at each vertex form an integral basis of $\Z^n$. These Delzant polytopes play an important role in toric geometry; each Delzant polytope $P$ constructs a toric manifold $X(P)$. Under the canonical Delzant realizations, both graph associahedra and graph cubeahedra become Delzant polytopes.
{For a graph~$G$, we will denote by $X^\R(\triangle_G)$ and $X^\R(\square_G)$ the real loci of the toric manifolds $X(\triangle_G)$ and $X(\square_G)$, respectively.}

The Betti numbers of the real toric manifold $X^\R(\triangle_G)$ are computed in~\cite{CP},  by a {graph invariant} called the \emph{$a$-number}. We write $H \sqsubseteq G$ when $H$ is an induced subgraph of a graph $G$, and $H\sqsubset G$ when $H$ is a proper induced subgraph of $G$. We say that a graph $G$ is \emph{even} (respectively, \emph{odd}) if every connected component of $G$ has an even (respectively, odd) number of vertices.
The $a$-number $a(G)$ of a graph $G$ is an integer defined as\footnote{The numbers $a(G)$ and $b(G)$ are originally defined in \cite{CP} in slightly different form. For a connected graph~$G$, $a(G)$ and $b(G)$ here  correspond to $sa(G)$ and $(-1)^{|V(G)|}b(G)$ in \cite{CP} respectively. See Lemma~\ref{lem:multiplicative}.}
    \[  a(G)=\begin{cases}
      1 &\text{if }V(G)=\emptyset,\\
      0 &\text{if $G$ is not even,}\\
      -\sum_{H: H\sqsubset G}  a(H) &\text{otherwise},\\
    \end{cases}\]
{and the $i$th Betti number  of $X^\R(\triangle_G)$ is given as follows:}
    \begin{theorem}[\cite{CP}]\label{thm:CP:main}
        Let $G$ be a graph. For any integer $i\ge 0$, the $i$th Betti number  of $X^\R(\triangle_G)$ is
        $$
        \beta^i(X^\R(\triangle_G))= \sum_{H \sqsubseteq G\atop |V(H)|=2i}\left|a(H)\right|.
        $$
    \end{theorem}

    Although graph cubeahedra are very similar to graph associahedra in their constructions, real toric manifolds corresponding to those ploytopes $\triangle_G$ and $\square_G$ are quite different since their normal fans are not isomorphic in general.
    So, it is natural to ask if analogous theories hold for real toric manifolds corresponding to graph cubeahedra and how they can be stated.
    In this paper, we first focus on the Betti numbers of the real toric manifolds $X^\R(\square_G)$.
    Let us define the \emph{$b$-number} of a graph $G$ as follows.
    \[  b(G)=\begin{cases}
      1 &\text{if }V(G)=\emptyset,\\
      0 &\text{if $G$ is not odd,}\\
      -\sum_{H: H\sqsubset G}  b(H) &\text{otherwise}.\\
    \end{cases}\]
We also obtain the $b$-number-analogue of Theorem~\ref{thm:CP:main} for $X^\R(\square_G)$ as follows.
\begin{theorem}\label{thm:main:typeI}
        Let $G$ be a graph. For any integer $i\ge 0$, the $i$th Betti number of $X^\R(\square_G)$ is
    $$
    \beta^i(X^\R(\square_G))= \sum_{H \sqsubseteq G\atop |V(H)|+\kappa(H)=2i}\left|b(H)\right|,
    $$
where $\kappa(H)$ is the number of connected components of $H$.
\end{theorem}

Recently, Manneville and Pilaud showed in~\cite{MP2017} that for two connected graphs $G$ and $G'$, there exists a combinatorial equivalence between a graph associahedron $\triangle_G$ and a graph cubeahedron $\square_{G'}$ if and only if $G$ is a tree with at most one vertex whose degree is greater than $2$ and $G'=L(G)$, where $L(G)$ is the line graph of $G$. Hence $h_i(\triangle_G)=h_i(\square_{L(G)})$ for $i=1,\ldots,|V(G)|-1,$ which implies that the toric manifolds $X(\triangle_G)$ and $X(\square_{L(G)})$ have the same Betti numbers. But $X(\triangle_G)$ and $X(\square_{L(G)})$ are not isomorphic as toric varieties even if $G$ is the path $P_4$, see Section~\ref{sec:polytope}.

For a tree $G$, even though the polytopes $\triangle_G$ and $\square_{L(G)}$ are not combinatorially equivalent in general, we get an interesting relationship between the real toric manifolds $X^\R(\triangle_G)$ and $X^\R(\square_{L(G)})$ as follows.
\begin{theorem}\label{prop:line:graph}
For a forest $G$, the real toric manifolds $X^\R(\triangle_G)$ and $X^\R(\square_{L(G)})$ have the same Betti numbers, that is, for any integer $i \ge 0$,
$$\beta^i(X^\R(\triangle_{G}))=  \beta^i(X^\R(\square_{L(G)})).$$
\end{theorem}
Note that the relation above does not hold in general, see Section~\ref{sec:main_results}.
Meanwhile obtaining the theorem above, we also discuss additional various properties related to $a(G)$ and $b(G)$ such as M\"obius inversion and Euler characteristics of $X^\R(\triangle_G)$ and $X^\R(\square_G)$.

This paper is organized as follows. In Section~\ref{sec:simple_graph}, we summarize the results of~\cite{CP} about real toric manifolds over graph associahedra.
In Section~\ref{sec:polytope}, we recall the definition of graph cubeahedra~$\square_G$, introduce the toric manifold $X(\square_G)$, and briefly review the relationship between $X(\triangle_G)$ and $X(\square_G)$.
Section~\ref{sec:main_results} states the main results, that is, we describe the Betti numbers of the real toric manifolds $X^\R(\square_G)$ in terms of {the} $b$-numbers of graphs {(Theorem~\ref{thm:main:typeI})} and then study various relationships between the real toric manifolds $X^\R(\triangle_G)$ and $X^\R(\square_{L(G)})$ including Theorem~\ref{prop:line:graph}.
Section~\ref{sec:proof} is devoted to the proof of  Theorem~\ref{thm:main:typeI}.
Section~\ref{sec:example} {provides some interesting integer sequences arising from the Betti numbers of the real toric manifolds associated with some special graphs. In Section~\ref{sec:remark}, we give some} remarks on a graph associahedron of type $B$.

\section{Real toric manifolds over graph associahedra}\label{sec:simple_graph}

In this section, we briefly summarize the result of \cite{CP}, which studies real toric manifolds over graph associahedra. {Recall that a real toric manifold is the real locus of a toric manifold and we refer the reader to~\cite{Ful1993} for more details of toric varieties.}

For a graph $G$ we denote by $\kappa(G)$ the number of connected components of a graph $G$, where a \emph{connected component} (or a \emph{component}) means a maximally connected subgraph of $G$.  The \emph{null graph} is the graph whose vertex set is empty, and it has no connected component and so $\kappa(G)$ for the null graph $G$ is defined to be $0$ by convention. We say that a graph $G$ is \emph{even} (respectively, \emph{odd}) if every connected component of $G$ has an even (respectively, odd) number of vertices. {Note that the null graph is a unique graph that is both even and odd.
A subgraph $H$ of $G$ is said to be \emph{induced} if $H$} includes the edges between every pair of vertices in $H$ if such edges exist in $G$. For $I\subset V(G)$, the subgraph induced by $I$ is denoted by $G[I]$,  {and for simplicity, we let $\mathcal{I}_G$ be the set of all $I\subset V(G)$ such that $G[I]$ is connected.}
Throughout this paper,  we denote by $H\sqsubseteq G$ {if $H$ is either an induced subgraph of $G$ or a null graph, and when $G$ is not the null graph, we denote by $H\sqsubset G$ if $H$ is either a proper induced subgraph of $G$ or a null graph.}
We denote a complete graph, a path, a cycle, and a star with $n$ vertices by $K_n$, $P_n$, $C_n$, and $K_{1,n-1}$, respectively.

\bigskip

\noindent\textbf{Construction of a graph associahedron.}
Let $G$ be a connected graph with the vertex set $[n]:=\{1,\ldots,n\}$.
Let us consider the standard simplex $\Delta^{n-1}$ whose facets are labeled by $1,\ldots,n$. Then there is a one-to-one correspondence between the faces of $\Delta^{n-1}$ and the subsets of $[n]$. Hence each face of $\Delta^{n-1}$ can be labeled by a subset $I\subset[n]$.
Then the graph associahedron, denoted by $\triangle_G$, is obtained from $\Delta^{n-1}$ by truncating
the faces labeled by $I$ for each proper connected subgraph $G[I]$ in increasing order of dimensions. If $G$ has the connected components $G_1,\ldots,G_\kappa$, then we define $\triangle_G=\triangle_{G_1}\times\cdots\times\triangle_{G_\kappa}$.

\bigskip

\begin{lemma}[\cite{CPP2015}]\label{lem:truncation and delzant}
    Let $P$ be a Delzant polytope and $F$ a proper face of $P$. Then there is {a} canonical truncation of $P$ along $F$ such that the result is a Delzant polytope, say $\mathrm{Cut}_F(P)$, satisfying that the toric manifold $X(\mathrm{Cut}_F(P))$ is the blow-up of $X(P)$ along the submanifold $X(F)\subset X(P)$.
\end{lemma}

{The lemma above assures that $\triangle_G$ is a well-defined Delzant polytope, and the toric manifold $X(\triangle_G)$ is an iterated blow-up of $X(\Delta^n)=\C P^n$. Note that if a face $F$ of $P$ is the intersection of the facets $F_1, \cdots, F_k$, then the normal vector of the new face of $\mathrm{Cut}_F(P)$ arising from the truncation is the sum of the normal vectors of $F_1,\ldots,F_k$.}

\begin{remark}
    Note that, for disconnected graphs, the definitions of graph associahedra in~\cite{CD2006}~and~\cite{P05}  do not coincide, and we follow the definition in \cite{P05}; for a given graph $G$, the polytope constructed in \cite{CD2006} is combinatorially equivalent to $\triangle_{G}\times\Delta^{\kappa(G)-1}$.
\end{remark}

\begin{example} For a path $P_3$, the graph associahedron $\triangle_{P_3}$ is a pentagon, {which is obtained from a triangle by truncating two vertices, see Figure~\ref{fig:ex graphassociahedron}. Since $X(\Delta^2)$ is the complex projective space $\C P^2$, the toric manifold $X(\triangle_{P_3})$ corresponds to blowing up of $\C P^2$ at two fixed points of the torus action, and hence $X(\triangle_{P_3})=\C P^2\#\overline{\C P^2}\#\overline{\C P^2}.$ Therefore, the real toric manifold $X^\R(\triangle_{P_3})$ is $\R P^2\#\R P^2 \# \R P^2$.}
    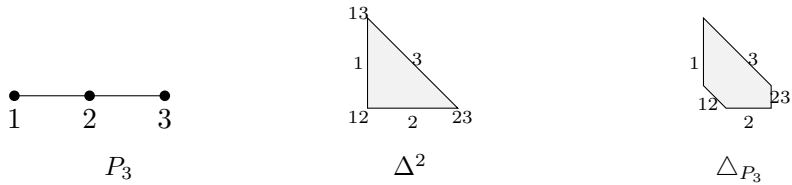
\begin{figure}[h]
    \begin{center}
     \begin{subfigure}[b]{.19\textwidth}       \begin{tikzpicture}[scale=1]
            \fill (0,0) circle(2pt);
            \fill (1,0) circle(2pt);    \fill (2,0) circle(2pt);
            \draw (0,0)--(2,0);
            \draw (0,-0.3) node{$1$};
            \draw (1,-0.3) node{$2$};
            \draw (2,-0.3) node{$3$};
        \end{tikzpicture}
                \caption*{$P_3$}
    \end{subfigure}
        \begin{subfigure}[b]{.25\textwidth}
        \centering
        \begin{tikzpicture}[scale=.6]
    		\filldraw[fill=gray!10] (0,0)--(2,0)--(0,2)--cycle;
    		\draw (-0.2,-0.2) node{{\tiny $12$}};
    		\draw (-0.2,2.1) node{{\tiny $13$}};
    		\draw (2.1,-0.2) node{{\tiny $23$}};
    		\draw (1,-0.3) node{{\tiny $2$}};
    		\draw (1.1,1.1) node{{\tiny $3$}};
    		\draw (-0.2,1) node{{\tiny $1$}};
    	\end{tikzpicture}
            \caption*{$\Delta^2$}
        \end{subfigure}
        \begin{subfigure}[b]{.25\textwidth}
        \centering
        \begin{tikzpicture}[scale=.6]
            \filldraw[fill=gray!10] (0,0.5)--(0.5,0)--(1.5,0)--(1.5,0.5)--(0,2)--cycle;
            \draw (0.1,0.1) node{{\tiny $12$}};
    		\draw (1.7,0.25) node{{\tiny $23$}};
    		\draw (1,-0.3) node{{\tiny $2$}};
    		\draw (1.1,1.1) node{{\tiny $3$}};
    		\draw (-0.2,1) node{{\tiny $1$}};
        \end{tikzpicture}
            \caption*{$\triangle_{P_3}$}
        \end{subfigure}
        \end{center}
        \caption{A graph associahedron $\triangle_{P_3}$.}\label{fig:ex graphassociahedron}
        \end{figure}
\end{example}

Note that the graph associahedra corresponding to a path $P_n$, a cycle $C_n$, a complete graph $K_n$, and a star  $K_{1,n-1}$, are called an associahedron, a cyclohedron, a permutohedron, and a stellohedron, respectively, and they are well-studied in various contexts such as \cite{P05} and \cite{PRW}.

The face structure of $\triangle_G$ can be described from the structure of the graph $G$.

\begin{proposition}[\cite{CD2006}]
    For a connected graph $G$,
    there is a one-to-one correspondence between the facets of $\triangle_G$ and the proper connected induced subgraphs of $G$. We denote by $F_I$ the facet corresponding to a proper connected induced subgraph $G[I]$. Furthermore, the facets $F_{I_1},\ldots,F_{I_k}$ intersect if and only if $I_i\subseteq I_j$, $I_j\subseteq I_i$, or {$I_i\cup I_j\not\in\mathcal{I}_G$} for all $1 \le i < j \le k$.
\end{proposition}

We can also write the (outward) primitive normal vector of each facet of $\triangle_G$ explicitly; when $G$ is a connected graph, for each facet $F_I$ corresponding to the proper connected induced subgraph $G[I]$, the primitive (outward) normal vector of $F_I$ is
\begin{equation*}
    \begin{cases}
        -\sum_{i\in I}\mathbf{e}_i,&\text{ if }n\not\in I,\\
        \sum_{j\not\in I}\mathbf{e}_j,&\text{ if }n\in I.
    \end{cases}
\end{equation*}

We restate Theorem~\ref{thm:CP:main} below, which is the main result of \cite{CP}.
\begin{theorem}[Theorem~1.1]
        Let $G$ be a graph. For any integer $i \ge 0$, the $i$th Betti number of the real toric manifold $X^\R(\triangle_G)$ is
        $$
            \beta^i(X^\R(\triangle_G))= \sum_{H\sqsubseteq G\atop |V(H)|=2i}\left|a(H)\right|.
        $$
    \end{theorem}

The Betti numbers of real toric manifolds associated with some interesting families of graphs are computed by using Theorem~\ref{thm:CP:main}.

\begin{corollary}[\cite{CP}]\label{cor:number}
Let $G$ be a graph with $n+1$ vertices.
For $1\le i\le\lfloor\frac {n+1}2\rfloor$,
$$
 \beta^i(X^\R(\triangle_G))=\begin{cases}
              \binom{n+1}{2i}A_{2i} & \text{if }G=K_{n+1},\\
              \binom{n+1}i-\binom{n+1}{i-1} &  \text{if }G=P_{n+1},\\
              \binom{n+1}i&\text{if }G=C_{n+1}\text{ and }2i< n+1,\\
              \frac 12\binom{2i}i &\text{if }G=C_{n+1}\text{ and }2i= n+1,\\
              \binom{n}{2i-1}A_{2i-1} &\text{if }G=K_{1,n},\\
            \end{cases}
        $$
where $A_k$ is the $k$th Euler zigzag number given by
    \[
        \sec x + \tan x = \sum_{k=0}^\infty A_{k}\frac{x^{k}}{k!}.
    \]
\end{corollary}
Refer to \cite{SS} for the formulae for $a(G)$ and $\beta^i(X^\R(\triangle_G))$ for a complete multipartite graph $G$.

\bigskip

We finish the section by noting flagness. A simple polytope is \emph{flag} if any set of pairwise intersecting facets has nonempty intersection.

\begin{proposition}[Corollary~7.2\cite{PRW}] \label{lem:flag}
    For a graph $G$, the graph associahedron $\triangle_G$ is flag.
\end{proposition}

\section{Graph cubeahedra} \label{sec:polytope}

In this section, we briefly review the construction of a graph cubeahedron in~\cite{DHV} and a relationship between graph associahedra and graph cubeahedra. Set $[\overline{n}]=\{\overline{1},\ldots,\overline{n}\}$.

\medskip
\noindent\textbf{Construction of a graph cubeahedron.}
Let us consider the standard cube $\square^n$ whose facets are labeled by $1,\ldots,n$ and $\overline{1},\ldots,\overline{n}$, where the two facets labeled by $i$ and $\overline{i}$ are on opposite sides. Then every face of $\square^n$ can be labeled by a subset $I$ of $[n]\cup[\overline{n}]$ satisfying {that
$I\cap [n]$ and $\{ i\in[n] \mid \bar{i} \in I\}$ are disjoint.}
Let $G$ be a graph with the vertex set $[n]$. {Recall that $\mathcal{I}_G$ is the set of all subsets $I$ of $[n]$ such that $G[I]$ is connected.}
The \emph{graph cubeahedron}, denoted by $\square_G$, is obtained from the standard cube $\square^n$ by truncating the faces labeled by $I\in\mathcal{I}_G$
in increasing order of dimensions. {It follows from Lemma~\ref{lem:truncation and delzant} that the graph cubeahedron~$\square_G$ is also a Delzant polytope, and the toric manifold $X(\square_G)$ is an iterated blow-up of $X(\square^n)=(\C P^1)^n$.}

\bigskip

\begin{example}
For paths $P_2$ and $P_3$, the graph cubeahedra $\square_{P_2}$ and $\square_{P_3}$ are illustrated in Figure~\ref{fig:ex:cube}. {Note that the toric manifold corresponding to the standard cube $\square^2$ is $\C P^1\times \C P^1$. Hence $X(\square_{P_2})$ corresponds to blow up of $\C P^1\times \C P^1$ at one fixed point of the torus action. Thus $X(\square_{P_1})=(\C P^1\times \C P^1)\#\overline{\C P^2}$ and the real toric manifold $X^\R(\square_{P_2})$ is $(\R P^1\times \R P^1)\#{\R P^2}$.}
    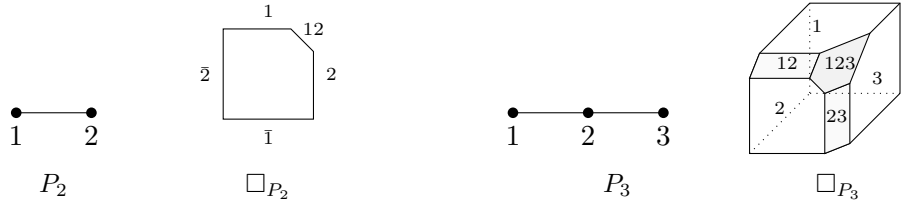
\begin{figure}[h]
    \centering
    \begin{subfigure}[b]{.16\textwidth}
    \centering
        \begin{tikzpicture}[scale=1]
            \fill (0,0) circle(2pt);
            \fill (1,0) circle(2pt);
            \draw (0,0)--(1,0);
            \draw (0,-0.3) node{$1$};
            \draw (1,-0.3) node{$2$};
        \end{tikzpicture}
        \caption*{$P_2$}
    \end{subfigure}
    \begin{subfigure}[b]{.16\textwidth}
    \centering
        \begin{tikzpicture}[scale=1.2]
            \draw (0,0)--(1,0)--(1,0.75)--(0.75,1)--(0,1)--cycle;
            \draw (0.5,-0.2) node{\tiny$\bar1$};
            \draw (0.5,1.2) node{\tiny$1$};
            \draw (-0.2,0.5) node{\tiny$\bar2$};
            \draw (1.2,0.5) node{\tiny$2$};
            \draw (1,1) node{\tiny$12$};
        \end{tikzpicture}
        \caption*{$\square_{P_2}$}
    \end{subfigure}
\qquad\qquad
     \begin{subfigure}[b]{.19\textwidth}       \begin{tikzpicture}[scale=1]
            \fill (0,0) circle(2pt);
            \fill (1,0) circle(2pt);    \fill (2,0) circle(2pt);
            \draw (0,0)--(2,0);
            \draw (0,-0.3) node{$1$};
            \draw (1,-0.3) node{$2$};
            \draw (2,-0.3) node{$3$};
        \end{tikzpicture}
                \caption*{$P_3$}
    \end{subfigure}
    \begin{subfigure}[b]{.14\textwidth}
        \begin{tikzpicture}[scale=.4]
            \draw (0,0)--(2.5,0)--(3+1/3,1/3)--(5,2)--(5,5)--(2,5)--(1/3,3+1/3)--(0,2.5)--cycle;
            \filldraw[fill=gray!4] (0,2.5)--(2,2.5)--(2+1/3,3+1/3)--(1/3,3+1/3)--cycle; 
            \filldraw[fill=gray!4] (2.5,0)--(3+1/3,1/3)--(3+1/3,2+1/3)--(2.5,2)--cycle; 
            \filldraw[fill=gray!10] (2+1/3,3+1/3)--(2,2.5)--(2.5,2)--(3+1/3,2+1/3)--(4,4)--cycle; 
            \draw[dotted] (2,5)--(2,2);
            \draw[dotted] (0,0)--(2,2);
            \draw[dotted] (5,2)--(2,2);
            \draw (4,4)--(5,5);
            \draw (1,1.5) node {\tiny$2$};
            \draw (2.25,4.25) node {\tiny$1$};
            \draw (1.25,3) node {\tiny$12$};
            \draw (3,3) node {\tiny$123$};
            \draw (2.9,1.25) node {\tiny$23$};
            \draw (4.25,2.5) node {\tiny$3$};
        \end{tikzpicture}
        \caption*{$\square_{P_3}$}
    \end{subfigure}\caption{Graph cubeahedra $\square_{P_2}$ and $\square_{P_3}$.}\label{fig:ex:cube}
    \end{figure}
\end{example}

Let us describe the facets of the graph cubeahedron and their outward normal vectors.
{We label each facet of $\square_G$ by $F_I$, where $I\in \mathcal{I}_G$
or $I$ is} a singleton subset of $[\bar{n}]$.
Then the primitive (outward) normal vector of the facet $F_I$ is
\begin{equation}\label{eq:Z_2-facet-vector}
    \begin{cases}
        \sum_{i\in I}\mathbf{e}_i&\text{ if }{I\in \mathcal{I}_G},\\
        -\mathbf{e}_i&\text{ if }I=\{\bar{i}\} \text{ for some }i\in [n].
    \end{cases}
\end{equation}

For a {graph} $G$ consisting of the connected components $G_1,\ldots,G_\kappa$, one easily shows that $\square_G$ is equivalent to $\square_{G_1}\times\cdots\times\square_{G_\kappa}$, where two Delzant polytopes are \emph{equivalent} if their normal fans are isomorphic.

Now we describe the face poset of $\square_G$ as in the following, which was given in   \cite{DHV}. The flagness is implicitly stated in  \cite{DHV} in describing its face poset.

\begin{proposition}[\cite{DHV}]\label{prop:faces_meet}
Let $G$ be a graph with the vertex set $[n]$. Then two facets $F_I$ and $F_J$ of $\square_G$  intersect if and only if one of the following holds.
        \begin{enumerate}
       \item Both $I$ and $J$ belong to $\mathcal{I}_G$ {and they satisfy} either
             $I\subseteq J$, $J\subseteq I$, or {$I\cup J\not\in\mathcal{I}_G$.}
       \item Exactly one of $I$ and $J$, say $I$, belongs to $\mathcal{I}_G$ and $J=\{\bar{j}\}$ for some $j\in [n]\setminus I$.
       \item Both $I$ and $J$ are singleton subsets of $[\bar{n}]$.
           \end{enumerate}
        Furthermore, the graph cubeahedron $\square_G$ is flag.
    \end{proposition}

We can easily check that the map from $2^{[n]}\cup\{\{\bar{1}\},\ldots,\{\bar{n}\}\}$ to $[n+1]$ defined by
\begin{equation*}
    I\mapsto\left\{\begin{array}{ll}
      \{i\} & \text{ for }I=\{\bar{i}\},\\
      {[n+1]} \setminus I & \text{ for }I\subset [n],
    \end{array}\right.
\end{equation*}
gives an isomorphism from the normal fan of $\square_{K_n}$ to the normal fan of $\triangle_{K_{1,n}}$, where $K_n$ is a complete graph and $K_{1,n}$ is a star. We can also easily check that the map from the set of facets of $\triangle_{P_{n+1}}$ to that of $\square_{P_n}$ given by
\begin{equation*}
    I \mapsto \left\{\begin{array}{ll}
        I & \text{ if } I\subset [n],\\
        \{\bar{j}\} & \text{ if }n+1\in I \text{ and } |I|={n+1-j},
    \end{array}
    \right.
\end{equation*} gives an isomorphism from the face poset of $\triangle_{P_{n+1}}$ to that of $\square_{P_n}$. However, there is no isomorphism between the normal fan of $\triangle_{P_4}$ and that of $\square_{P_3}$ because $\triangle_{P_4}$ has a pair of square facets whose normal vectors are parallel but $\square_{P_3}$ has no pair of such square facets.

The relationship above was noted in~\cite{MP2017} between two polytopes  $\triangle_G$ and $\square_{H}$ when $G$ is an octopus and  $H$ is a spider. An \emph{octopus} is a tree with at most one vertex of degree more than two. A \emph{spider} is a graph obtained from a complete graph $K_n$ by attaching at most one path by {one of} its leaf to each vertex of $K_n$, see Figure~\ref{fig:spider}.
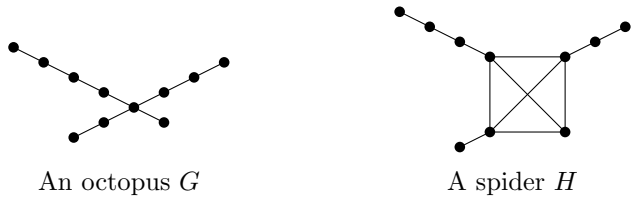
\begin{figure}[h]
    \centering
    \begin{subfigure}[b]{.3\textwidth}
    \centering
        \begin{tikzpicture}[scale=1]
            \path (0,0) coordinate (0)
                  (-0.4,0.2) coordinate (01)
                  (-0.8,0.4) coordinate (02)
               (-1.2,0.6) coordinate (03)
               (-1.6,0.8) coordinate (04)
                (0.4,0.2) coordinate (11)
                (0.8,0.4) coordinate (12)        (1.2,0.6) coordinate (13)
                (-0.4,-0.2) coordinate (21)     (-0.8,-0.4) coordinate (22)
                (0.4,-0.2) coordinate (31)
                              (0.4,-0.2) coordinate (41)
                  (0.8,-0.4) coordinate (42)
               (1.2,-0.6) coordinate (43);
            \fill (01) circle(2pt);
            \fill  (02) circle(2pt) ;
            \fill (03) circle(2pt);       \fill (04) circle(2pt);
            \fill
            (11) circle(2pt)  ;
            \fill(12) circle(2pt);   \fill(13) circle(2pt);
            \fill  (21) circle(2pt);            \fill  (22) circle(2pt);
            \fill (0) circle(2pt);      \fill (31) circle(2pt);
            \draw (0)--(04);      \draw (0)--(13);
              \draw (0)--(22);      \draw (0)--(31);
        \end{tikzpicture}
        \caption*{An octopus $G$}
    \end{subfigure}
    \begin{subfigure}[b]{.3\textwidth}
    \centering
           \begin{tikzpicture}[scale=1]
            \path (0,0) coordinate (0)
                  (-0.4,0.2) coordinate (01)
                  (-0.8,0.4) coordinate (02)
               (-1.2,0.6) coordinate (03)
                (1.4,0.2) coordinate (11)
                (1.8,0.4) coordinate (12)
                (-0.4,-1.2) coordinate (21);
            \fill (01) circle(2pt);
            \fill  (02) circle(2pt)  ;
            \fill(03) circle(2pt);
            \fill
            (11) circle(2pt) ;
            \fill (12) circle(2pt) ;
            \fill (21) circle(2pt) ;
            \fill (0) circle(2pt);   \fill (1,0) circle(2pt);\fill (0,-1) circle(2pt);\fill (1,-1) circle(2pt);
            \draw (0)--(03);      \draw (1,0)--(12);      \draw (0,-1)--(21);    \draw (0,-1)--(0)--(1,0)--(1,-1)--cycle; \draw (0)--(1,-1);\draw (1,0)--(0,-1);
        \end{tikzpicture}
                \caption*{A spider $H$}
    \end{subfigure}\caption{An octopus $G$ and a spider $H$.}\label{fig:spider}
    \end{figure}
The \emph{line graph} $L(G)$ of a graph $G$ is the intersection graph of $E(G)$. In other words, the vertex set of $L(G)$ is $E(G)$ and two vertices $e$ and $e'$ of $L(G)$ are adjacent if and only if $e\cap e'\neq \emptyset$,
{that is, $e$ and $e'$ share an endpoint in $G$.} Note that the line graph of an octopus is a spider. In Figure~\ref{fig:spider}, the line graph of $G$ is equal to $H$.

\begin{proposition}[\cite{MP2017}]\label{prop:spider:octopus}
For two connected graphs $G$ and $H$, the polytopes  $\triangle_G$ and $\square_{H}$ are combinatorially equivalent if and only if $G$ is an octopus and $H$ is a spider which is equal to the line graph of $G$. Furthermore, if $G$ is a star $K_{1,n}$, the normal fan of $\triangle_G$ is isomorphic to the normal fan of $\square_{L(G)}$.
\end{proposition}

Thus for an octopus $G$, the $h$-vector of $\triangle_G$ is equal to that of $\square_{L(G)}$, and hence the Betti numbers of the toric manifold $X(\triangle_G)$ are equal to {those} of the toric manifold $X(\square_{L(G)})$, and the $\Z_2$-Betti numbers of the real toric manifold $X^\R(\triangle_G)$ are equal to {those} of the real toric manifold $X^\R(\square_{L(G)})$.

\begin{remark}
    In~\cite{MP2017}, the authors define a family of complete simplicial fans, called \emph{compatibility fans}, whose underlying simplicial complex is {dual to} the graph associahedron, and they also define a family of complete simplicial fans, called \emph{design compatibility fans}, whose underlying simplicial complex is {dual to} the graph cubeahedron. The normal fan of $\triangle_G$ (respectively, $\square_G$) is a compatibility fan (respectively, a design compatibility fan) associated with $G$ and the normal fan of $\triangle_G$ is not isomorphic to that of $\square_{L(G)}$ even if $G$ is an octopus in general. But the normal fan of $\square_{L(G)}$ is isomorphic to some complete non-singular fan associated with $\triangle_{G}$, the dual compatibility fan, for an octopus~$G$. See~\cite{MP2017} for more details.
\end{remark}

\section{The $a$-number and the $b$-number of a graph and their relationships} \label{sec:main_results}

In this section, {we first study how the $a$-numbers and the $b$-numbers are related to each other and then describe the Betti numbers of the real toric manifold corresponding to a graph $G$ in terms of the $b$-numbers. We also show that the real toric manifolds $X^\R(\triangle_G)$ and $X^\R(\square_{L(G)})$ have the same Betti numbers for a forest~$G$.}

Recall the two graph invariants $a(G)$ and $b(G)$ given in Section~\ref{sec1}.
For a graph $G$,  $a(G)$ and $b(G)$, called the \emph{$a$-number} and the \emph{$b$-number} of $G$, respectively, are defined as follows.
    \[  a(G)=\begin{cases}
      1 &\text{if }V(G)=\emptyset,\\
      0 &\text{if $G$ is not even},\\
      -\sum_{H: H\sqsubset G}  a(H) &\text{otherwise},\\
    \end{cases} \quad  \quad b(G)=\begin{cases}
      1     &\text{if }V(G)=\emptyset,\\
      0 &\text{if $G$ is not odd},\\
      -\sum_{H: H\sqsubset G}  b(H) &\text{otherwise}.\\
    \end{cases}\]

    \begin{example}
        Let us compute the $a$- and $b$-numbers of a path $G=P_4$, where $V(P_4)=\{1,2,3,4\}$ and $E(P_4)=\{\{1,2\},\{2,3\},\{3,4\}\}$. Then the $a$- and $b$-number of the induced subgraphs of $P_4$ are as follows:
        \begin{equation*}
            \begin{array}{l}
                a(G[1])=a(G[2])=a(G[3])=a(G[4])=a(P_1)=0\\
                a(G[1,2])=a(G[2,3])=a(G[3,4])=a(P_2)=-1\\
                a(G[1,3])=a(G[1,4])=a(G[2,4])=a(P_1\sqcup P_1)=0\\
                a(G[1,2,3])=a(G[2,3,4])=a(P_3)=0\\
                a(G[1,2,4])=a(G[1,3,4])=a(P_1\sqcup P_2)=0\\
                a(P_4)=-a(\emptyset)-3a(P_2)=2,
            \end{array}
        \end{equation*} and
        \begin{equation*}
            \begin{array}{l}
                b(G[1])=b(G[2])=b(G[3])=b(G[4])=b(P_1)=-1\\
                b(G[1,2])=b(G[2,3])=b(G[3,4])=b(P_2)=0\\
                b(G[1,3])=b(G[1,4])=b(G[2,4])=b(P_1\sqcup P_1)=1\\
                b(G[1,2,3])=b(G[2,3,4])=b(P_3)=-b(\emptyset)-3b(P_1)-b(P_1\sqcup P_1)=1\\
                b(G[1,2,4])=b(G[1,3,4])=b(P_1\sqcup P_2)=0\\
                b(P_4)=0.
            \end{array}
        \end{equation*}
    \end{example}

    One can observe that the invariants $a(G)$ and $b(G)$ are the M\"obius invariants of some bounded posets as follows. A poset $\mathcal{P}$ is \emph{bounded} if it has a unique maximum element, denoted by $\hat{1}$, and a unique minimum element, denoted by $\hat{0}$. For a finite bounded poset $\mathcal{P}$, the \emph{M\"{o}bius invariant} of  $\mathcal{P}$ is defined as $\mu(\mathcal{P})=\mu_{\mathcal{P}}(\hat{0},\hat{1})$.\footnote{
The M\"{o}bius function $\mu$ can be defined inductively by the following relation: for a finite poset $\mathcal{P}$ and $s,t\in \mathcal{P}$,
    \[
        \mu_\mathcal{P}(s,t) =
        \begin{cases}
            {}\qquad 1 & \textrm{if}\quad s = t\\[6pt]
            \displaystyle -\sum_{r\, :\,  s\leq r <t} \mu_\mathcal{P}(s,t) & \textrm{for} \quad s<t \\[6pt]
            {}\qquad 0 & \textrm{otherwise}.
        \end{cases}
    \]}

    For a graph $G$, we define
\begin{eqnarray*}
  \mathcal{P}_G^{\mathrm{even}}&=&\{\emptyset\neq I\subsetneq V(G)\mid G[I] \text{ is even} \}\cup\{\hat{0},\hat{1}\} \text{ and}\\
    \mathcal{P}_G^{\mathrm{odd}}&=&\{\emptyset\neq I\subsetneq V(G)\mid G[I] \text{ is odd} \}\cup\{\hat{0},\hat{1}\}.
\end{eqnarray*}
By the definitions of $a$- and $b$-numbers, $a(G)=\mu({\mathcal{P}_G^{\mathrm{even}}})$ when $G$ is even, and $b(G)=\mu({\mathcal{P}_G^{\mathrm{odd}}})$ when $G$ is odd.
We can also check that the $a$- and $b$- numbers are multiplicative as follows.
\begin{lemma}\label{lem:multiplicative}
    Let $G$ and $H$ be two disjoint graphs. Then we have
    \[
        a(G \sqcup H) = a(G)a(H)\quad \text{and} \quad b(G \sqcup H) = b(G)b(H).
    \]
\end{lemma}
\begin{proof}
If $G\sqcup H$ is not even, then both $a(G \sqcup H)$ and $a(G)a(H)$ are zero by definition. Similarly, if $G\sqcup H$ is not odd, then both $b(G \sqcup H)$ and $b(G)b(H)$ are zero by definition. If $G\sqcup H$ is even (respectively, odd), then we has an isomorphism $\mathcal{P}_{G\sqcup H}^{\mathrm{even}} \cong \mathcal{P}_G^{\mathrm{even}} \times \mathcal{P}_H^{\mathrm{even}}$ (respectively, $\mathcal{P}_{G\sqcup H}^{\mathrm{odd}} \cong \mathcal{P}_G^{\mathrm{odd}} \times \mathcal{P}_H^{\mathrm{odd}}$). Therefore, the proof is done by multiplicativity of M\"obius invariants.
\end{proof}

A finite, pure simplicial complex $K$ of dimension $n$ is called \emph{shellable} if there is an ordering $C_1, C_2, \ldots,C_t$ of the maximal simplices of $K$, called a \emph{shelling}, such that $(\bigcup_{i=1}^{k-1}C_i)\cap C_k$ is pure of dimension $n-1$ for every $k=2,3,\dotsc,t$. It is well-known in~\cite{S} that shellable complexes are Cohen-Macaulay and thus homotopy equivalent to a wedge of spheres of the same dimension. In \cite{B}, Bj\"orner presented a criterion for shellability of order complexes. If the order complex of a poset $\mathcal{P}$ is shellable, then we say $\mathcal{P}$ is shellable.

\begin{theorem}[\cite{CP}]\label{prop:simple-shellable}
    For every graph $G$, $\mathcal{P}_G^{\mathrm{even}}$ is a pure shellable poset of length $\left\lceil\frac{|V(G)|}{2}\right\rceil$.
\end{theorem}
Hence the order complex of $\mathcal{P}_G^{\mathrm{even}}\setminus\{\hat{0},\hat{1}\}$ is homotopy equivalent to the wedge of $|\mu|$ copies of the spheres $S^d$, where $\mu=\mu(\mathcal{P}_G^{\mathrm{even}})$ and $d={\left\lceil\frac{|V(G)|}{2}\right\rceil-2}$.
In fact, the $a$-number and the $b$-number determine each other as follows.

\begin{theorem}\label{thm:abdual}
    For every graph $G$, we have
    \begin{equation}\label{eqn:bwrittenbya}
        b(G) = (-1)^{|V(G)|}\sum_{H:H\sqsubseteq G}a(H)
    \end{equation}
    and
    \begin{equation}\label{eqn:awrittenbyb}
        a(G) = \sum_{H:H\sqsubseteq G}b(H).
    \end{equation}
\end{theorem}

\begin{proof}
    Let us prove \eqref{eqn:bwrittenbya} first. If $G$ is a connected even graph, then $b(G)=0$ and $a(G)=-\sum_{H\sqsubset G}a(H)$ from their definitions, and hence \eqref{eqn:bwrittenbya} holds. {The formula \eqref{eqn:bwrittenbya} trivially holds when $|V(G)|=1$.} Now we assume that $G$ is a connected graph with $2k+1$ vertices {for $k\ge 1$}. Recall that every facet of $\triangle_G$ is labeled {by} $I$ such that $G[I]$ is a proper connected induced subgraph and we write that facet by $F_I$. Then we have
    \[
        \partial \triangle_G = \left(\bigcup_{|I|=\text{even}}F_I\right) \cup \left(\bigcup_{|I|=\text{odd}}F_I\right).
    \]
    The former set $\bigcup_{|I|=\text{even}}F_I$ is homotopy equivalent to $\mathcal{P}_G^{\mathrm{even}} \setminus \{\hat{0},\hat{1}\}$, and the latter set $\bigcup_{|I|=\text{odd}}F_I$ is homotopy equivalent to $\mathcal{P}_G^{\mathrm{odd}} \setminus \{\hat{0},\hat{1}\}$. Actually, {since $\triangle_G$ is a simple polytope, the dual of the set $\bigcup_{|I|=\text{even}}F_I$ (respectively, $\bigcup_{|I|=\text{odd}}F_I$) is a simplicial complex and it becomes the order complex of $\mathcal{P}_G^{\mathrm{even}} \setminus \{\hat{0},\hat{1}\}$ (respectively, $\mathcal{P}_G^{\mathrm{odd}} \setminus \{\hat{0},\hat{1}\}$) after suitable subdivisions, see Lemma~4.7 of \cite{CP}.\footnote{Note that this property holds only when $G$ is connected.}}
    Note that $\partial \triangle_G$ is homeomorphic to a sphere of dimension $2k-1$ and $\mathcal{P}_G^{\mathrm{even}}$ is {a} shellable {poset of length $k+1$} by Theorem~\ref{prop:simple-shellable}. Hence  we obtain
    \[
        b(G) = \mu({\mathcal{P}_G^{\mathrm{odd}}}) = \mu({\mathcal{P}_G^{\mathrm{even}}}) = (-1)^{|V(G)|}\sum_{H:H\sqsubseteq G}a(H),
    \]
    where the second identity follows from the Philip Hall theorem\footnote{For any bounded poset $\mathcal{P}$, the reduced Euler characteristic of the order complex of $\mathcal{P}\setminus \{\hat{0},\hat{1}\}$ is equal to the M\"obius invariant $\mu(\mathcal{P})$.} and the Alexander duality\footnote{The Alexander duality says that if $X$ is a compact, locally contractible subspace of a sphere $S^n$, then $\tilde{H}_q(X)\cong \tilde{H}^{n-q-1}(S^n\setminus X)$ for every $q$.} on the sphere $\partial \triangle_G$.

    When $G$ has the connected components $G_1,\ldots,G_\kappa$, then Lemma~\ref{lem:multiplicative} implies
    \[
        \prod_{i=1}^\kappa (-1)^{|V(G_i)|}\sum_{H:H\sqsubseteq G_i}a(H) = (-1)^{|V(G)|}\sum_{H:H\sqsubseteq G}a(H),
    \]
    {which proves \eqref{eqn:bwrittenbya}.}

    Now let us show \eqref{eqn:awrittenbyb}. {Let $\mathcal{P}$ be the poset of all elements $H\sqsubseteq G$. Note that $\mathcal{P}$ is isomorphic to the Boolean algebra, the poset of all subsets of $[n]$.} We apply the M\"obius inversion formula\footnote{See Proposition~3.7.1 of \cite{S2012}.} to~\eqref{eqn:bwrittenbya}. Then we immediately obtain that
\[ a(G) = \sum_{H:H\sqsubseteq G}(-1)^{|V(H)|}b(H)\mu_{\mathcal{P}}(H,G).\]
Since $\mu_{\mathcal{P}}(H,G) = (-1)^{|V(G)|-|V(H)|}$
  for any induced subgraph $H$ of $G$,
 it holds that
   $$a(G) = (-1)^{|V(G)|}\sum_{H:H\sqsubseteq G}b(H).$$
{Note that $a(G)=0$ whenever $|V(G)|$ is odd. Therefore, \eqref{eqn:awrittenbyb} holds.}
\end{proof}

\begin{remark}
In many cases, \eqref{eqn:bwrittenbya} is more efficient to compute the $b$-number than the definition, since usually there are fewer even induced subgraphs than odd ones. In practice one can use $b(G) = \mu({\mathcal{P}_G^{\mathrm{even}}})$ for computation when $G$ is a connected odd graph.
\end{remark}

In fact, the signs of $a(G)$ and $b(G)$ are completely determined by the graph~$G$.

\begin{corollary}\label{cor:absign}
    For a graph $G$, the signs of $a(G)$ and $b(G)$ are determined as follows.
    \begin{enumerate}
        \item If $G$ is even, then  $a(G) = (-1)^{\frac{|V(G)|}{2}}|a(G)|$.
        \item If $G$ is odd, then $b(G)= (-1)^{\frac{|V(G)|+\kappa(G)}{2}}|b(G)|$.
    \end{enumerate}
\end{corollary}
\begin{proof}
{Note that (1) is already known in \cite{CP}. Let us prove (2). It is well-known that the M\"obius function of $\mathcal{P}_G^{\mathrm{even}}$ alternates in sign\footnote{See Proposition~3.8.11 of \cite{S2012}.}. Let us write $\sgn x = x/|x|$ for any nonzero real number $x$. By \eqref{eqn:bwrittenbya},
we have
\begin{eqnarray}\label{eq:sign}
&&\sgn b(G) = (-1)^{|V(G)|}\times (-1)\times \sgn \mu(\mathcal{P}_G^{\mathrm{even}}).
\end{eqnarray}
    For a maximal element $H$ of $\mathcal{P}_G^{\mathrm{even}}\setminus\{\hat{1}\}$, $H$ is even and so by (1), we have $\sgn \mu_{\mathcal{P}_G^{\mathrm{even}}}(\hat{0}, H) = (-1)^{\frac{|V(H)|}{2}}$.
Thus the sign of $\mu(\mathcal{P}_G^{\mathrm{even}})$ is equal to that of $(-1)^{\frac{|V(H)|}{2}+1}$.
    From the fact that $|V(H)| = |V(G)|-\kappa(G)$,  \eqref{eq:sign} is equal to
    $$(-1)^{|V(G)|+1+\frac{|V(H)|}{2}+1}=(-1)^{|V(G)|+\frac{|V(G)|-\kappa(G)}{2}}=(-1)^{\frac{3|V(G)|-\kappa(G)}{2}}
    = (-1)^{\frac{|V(G)|+\kappa(G)}{2}},$$
where the last equality is from the fact that $3|V(G)|-\kappa(G)$ and $|V(G)|+\kappa(G)$   have the same parity.}
\end{proof}

{As the Betti numbers of $X^\R(\triangle_G)$ were formulated by the $a$-numbers, now we formulate the Betti numbers of $X^\R(\square_G)$ by the $b$-numbers. We restate the main theorem below, and its proof will be presented in Section~\ref{sec:proof}.}
\begin{theorem}[Theorem~\ref{thm:main:typeI}]
        Let $G$ be a graph.  For any integer $i\ge 0$, the $i$th Betti number of $X^\R(\square_G)$ is
    $$
    \beta^i(X^\R(\square_G))= \sum_{H \sqsubseteq G\atop |V(H)|+\kappa(H)=2i}\left|b(H)\right|.
    $$
\end{theorem}

The following implies that $a(G)$ and $b(G)$ are completely determined by the combinatorial structure of $\square_G$ and $\triangle_G$, respectively.
\begin{corollary}
    For any graph $G$, we have
    \[
        a(G) = \chi(X^\R(\square_G)) = \sum_i (-1)^i h_i(\square_G)
    \]
    and
    \[
        b(G) = (-1)^{|V(G)|}\chi(X^\R(\triangle_G)) = (-1)^{|V(G)|}\sum_i (-1)^i h_i(\triangle_G).
    \]
    where  $(h_0(P),h_1(P),\dotsc,h_n(P))$ is the $h$-vector of the simple polytope $P$.
\end{corollary}

\begin{proof}
    The second formula is already known in \cite{CP}. The first one is induced by the following chain of identities
    \[
        \chi(X^\R(\square_G)) = \beta^0-\beta^1+\beta^2- \cdots = \sum_{H\sqsubseteq G}(-1)^{\frac{|V(H)|+\kappa(H)}2}|b(H)| = \sum_{H\sqsubseteq G} b(H) = a(G);
    \]
    the first equality is by definition of the Euler characteristic, the second one is by Theorem~\ref{thm:main:typeI}, the third one is by Corollary~\ref{cor:absign}, and the last identity is by Theorem~\ref{thm:abdual}.  In both formulae, the parts containing $h_i$ are shown using the $\Z_2$-Betti numbers of the real toric manifolds and the fact that the Euler characteristic is independent from the choice of coefficient field.
\end{proof}

{From now on, we will see a significant application of our main result.
Let us take a look into a result in~\cite{MP2017} again}.
By Proposition~\ref{prop:spider:octopus}, if $G$ is an octopus and $L(G)$ is the corresponding spider, the line graph of $G$, then $h_i(\triangle_G)=h_i(\square_{L(G)})$ for any $i$, and hence the $\Z_2$-Betti numbers of the real toric manifolds $X^\R(\triangle_G)$ and $X^\R(\square_{L(G)})$ are the same. We can show that this phenomenon holds for the Betti numbers of $X^\R(\triangle_G)$ and $X^\R(\square_{L(G)})$ for any tree $G$ and its line graph $L(G)$.
We also note that the family of the line graphs of trees is one of important graph families in graph theory, so called claw-free block graphs.
In the rest of the section, we will prove the following {by} using Theorem~\ref{thm:main:typeI}, the main result.

\begin{theorem}[Theorem~\ref{prop:line:graph}]
For a forest $G$, the real toric manifolds $X^\R(\triangle_G)$ and $X^\R(\square_{L(G)})$ have the same Betti numbers, that is, for any integer $i \ge 0$, {we have}
$$\beta^i(X^\R(\triangle_{G}))=  \beta^i(X^\R(\square_{L(G)})).$$
\end{theorem}

Note that the theorem above does not hold in general. {For example, the line graph of a cycle $C_n$ is isomorphic to $C_n$ itself, but the Betti numbers of $X^\R(\triangle_{C_n})$ are different from those of $X^\R(\square_{C_n})$, see Corollary~\ref{cor:number} and Corollary~\ref{cor:bofcycle}.}
Before proving Theorem~\ref{prop:line:graph}, we will see an interesting identity which shows a new relationship between {the} $a$- and $b$-numbers.
The following lemma collects simple observations. A \textit{spanning subgraph} $H$ of $G$ is a subgraph of $G$ such that $V(H)=V(G)$.
For any graph $G$, let $\mathcal{S}(G)$ be the set of all spanning subgraphs of a graph $G$ without isolated vertices.

\begin{lemma}\label{for:a-b:tree:0}
For a forest $G$, the following hold.
\begin{itemize}
\item[(1)] For each even subgraph $H$ of $G$, $L(H)$ is an odd subgraph of $L(G)$.
\item[(2)] For each $H\in \mathcal{S}(G)$, $|V(L(H))|+\kappa(L(H))=|V(G)|.$
\end{itemize}
\end{lemma}

\begin{proof}
If $H$ is an even subgraph of  $G$, then
each component of $H$ has an odd number of edges and so each component of $L(H)$ has an odd number of vertices, which implies that  $L(H)$ is odd. Thus (1) holds.

Take any $H\in \mathcal{S}(G)$.
Since $H$ is also a forest, we have $|E(H)|+\kappa(H)=|V(H)|$. Note that
    $|V(L(H))|=|E(H)|$ and  $|V(H)|=|V(G)|$.
Since $H$ has no isolated vertex, $\kappa(L(H))=\kappa(H)$. Thus (2) follows.
\end{proof}

The following proposition is not only for proving Theorem~\ref{prop:line:graph} but also for providing a significant observation in a relationship between the $a$- and $b$-numbers.

\begin{proposition}\label{prop:line:graph:a-b}
For any even forest $G$,
$$ a(G)=\sum_{H\in\mathcal{S}(G)} b(L(H)) \qquad \text{ and  }\qquad |a(G)|=\sum_{H\in\mathcal{S}(G)} |b(L(H))|.$$
\end{proposition}

\begin{proof} Note that from the first equality, the second one follows immediately, since $a(G)$ and $b(L(G))$'s have the same sign by Corollary~\ref{cor:absign} and (2) of Lemma~\ref{for:a-b:tree:0}.
We prove the first equality by induction on $|V(G)|$. It is clear for $|V(G)|=2$. Now assume that the proposition is true for any even forest with at most $n-2$ vertices. Now take any even forest $G$ with $n$ vertices. By (1) of Lemma~\ref{for:a-b:tree:0}, $L(G)$ is an odd graph. Thus by the definition of {the} $b$-number,
    \begin{equation*}\label{eq:a-b-1}
        b(L(G))=-\sum_{L\colon L\sqsubset L(G)} b(L).
      \end{equation*}
For each subgraph $L$ of $L(G)$, we denote by $H_L$ the {minimal subgraph of $G$ whose edges are the elements in $V(L)$, that is, $E(H_L)=V(L)$.}
Then
\begin{equation*}
         b(L(G)) =-
         \sum_{L\colon L\sqsubset L(G)\atop V(H_L)=V(G)} b(L) -\sum_{L\colon L\sqsubset L(G)\atop {V(H_L)\subsetneq V(G)}} b(L).
\end{equation*}
Hence we have
    \begin{align}\label{eq:a-b-2}
        b(L(G))+\sum_{L\colon L\sqsubset L(G) \atop V(H_L)=V(G)} b(L) &= -\sum_{L\colon L\sqsubset L(G)\atop V(H_L)\subsetneq V(G)} b(L).
    \end{align}
    Then since for each $L\sqsubset L(G)$, any component of $H_L$ is not an isolated vertex,
        the left-hand-side of~\eqref{eq:a-b-2} is equal to $\sum_{H\in\mathcal{S}(G)} b(L(H))$ by definition.

We will show that the right-hand-side of~\eqref{eq:a-b-2} is equal to $a(G)$.
If $H$ is even  and $|V(H)|\le n-2$, then it follows from the induction hypothesis that
\begin{align}\label{eq:a-b-3}
a(H)&=\sum_{L\colon L\sqsubseteq L(H)\atop H_{L}\in\mathcal{S}(H)} b(L).
    \end{align}
Even if  $H$ is not even, we still have the same identity \eqref{eq:a-b-3}.
To see why, suppose that $H$ is not even.
Then  $a(H)=0$ and $H$ has a component with an odd number of vertices.
Since  $H_{L}\in \mathcal{S}(H)$ and so $H_{L}$ has no isolated vertex, $H_{L}$ must have a component with an even number of edges. Thus $L(H_{L})=L$ has a component with an even number of vertices, and so $b(L)=0$.
Therefore, the left and right hand sides of \eqref{eq:a-b-3} are equal to 0.

Thus the right-hand-side of~\eqref{eq:a-b-2} is equal to
$$-\sum_{H\colon H\sqsubset G} \quad \sum_{L\colon L\sqsubseteq L(H)\atop H_{L}\in\mathcal{S}(H)} b(L) =-\sum_{H\colon H\sqsubset G}a(H) =a(G)$$
where the first equality is from  \eqref{eq:a-b-3} and  the last one is from the fact that $G$ is even. It proves the first equality, and so completes the proof.
\end{proof}

By Theorems~\ref{thm:CP:main} and~\ref{thm:main:typeI}, and Proposition~\ref{prop:line:graph:a-b}, we can prove Theorem~\ref{prop:line:graph} as follows.
\begin{proof}[Proof of Theorem~\ref{prop:line:graph}]
\begin{equation*}
        \begin{split}
            \beta^i(X^\R(\triangle_G))&=\sum_{H\colon H\sqsubseteq G\atop |V(H)|=2i} |a(H)| \qquad \quad \qquad\qquad\text{ (by Theorem~\ref{thm:CP:main})}\\
            &=\sum_{H\colon H\sqsubseteq G\atop |V(H)|=2i}\sum_{H'\in\mathcal{S}(H)} |b(L(H'))| \qquad\text{ (by~Proposition~\ref{prop:line:graph:a-b})}\\
            &=\sum_{L\colon L\sqsubseteq L(G)\atop |V(L)|+\kappa(L)=2i} |b(L)| \quad \qquad\qquad\quad\text{ (by~Lemma~\ref{for:a-b:tree:0})}\\
            &=\beta^i(X^\R(\square_{L(G)}))\qquad\qquad\qquad\qquad {\text{(by Theorem~\ref{thm:main:typeI})}}.
        \end{split}
    \end{equation*}
\end{proof}

\section{Proof of Theorem~\ref{thm:main:typeI}}\label{sec:proof}

In this section, we first prepare some definitions and known results to prove our main theorem, and then give the proof of Theorem~\ref{thm:main:typeI}.

\bigskip

\noindent{\textbf{The cohomology of a real toric manifold.}} \
{We present a result on
the cohomology groups of  a real toric manifold introduced in \cite{CaiChoi,CP2,CP3,ST2012,Tre2012}.}
Let $P$ be a Delzant polytope of dimension~$n$ and let $\cF(P)=\{F_1,\ldots,F_m\}$ be the set of facets of $P$. Then the primitive outward normal vectors of $P$ can be understood as a function~$\phi$ from $\cF(P)$ to $\Z^n$, and the composition map $\lambda \colon \cF(P) \stackrel{\phi}{\rightarrow} \Z^n \stackrel{\text{mod $2$}}{\longrightarrow} \Z_2^n$ is called the (mod $2$) \emph{characteristic function} over $P$.  Note that $\lambda$ can be represented by a $\Z_2$-matrix $\Lambda_P$ of size $n \times m$ as
    $$
    \Lambda_P = \begin{pmatrix}
      \lambda(F_1) & \cdots & \lambda(F_m)
    \end{pmatrix},
    $$ where the $i$th column of $\Lambda_P$ is $\lambda(F_i) \in \Z_2^n$.
    For $\omega \in \Z_2^m$, we define $P_\omega$ to be the union of facets $F_j$ such that the $j$th entry of $\omega$ is nonzero.
    Then the following holds:
  \begin{theorem}[\cite{ST2012,Tre2012}]\label{formula}
    Let  $P$ be a Delzant polytope of dimension $n$. Then the $i$th Betti number of the real toric manifold $X^\R(P)$ is given by
        $$
        \beta^i (X^\R(P)) = \sum_{S\subseteq [n]} \tilde{\beta}^{i-1}(P_{\omega_S}),
        $$
        where $\omega_S$ is the sum of the $k$th rows of $\Lambda_P$ for all $k\in S$.
    \end{theorem}

    It is shown in~\cite{CaiChoi} that the cohomology group of a real toric manifold $X^\R(P)$ is completely determined by the reduced cohomology groups of $P_{\omega_S}$'s and the $h$-vector of $P$. In particular, if  $\tilde{H}^\ast(P_{\omega_S})$ is torsion-free for every $S\subseteq [n]$, then the cohomology group of $X^\R(P)$ is
    \begin{equation}\label{eq:CaiChoi}
        H^{i}(X^\R(P))\cong \Z^{\beta^i}\oplus \Z_2^{h_i-\beta^i},
    \end{equation} where $\beta^i$ is the $i$th Betti number of $X^\R(P)$ and $(h_0,h_1,\ldots,h_n)$ is the $h$-vector of $P$.

\bigskip
\noindent\textbf{The $\boldsymbol{\Z}_{\mathbf{2}}$-characteristic matrix of the real toric manifold $\boldsymbol{X^\R(\square_G)}$.}
Let $G$ be a graph with the vertex set $[n]$. Recall that
$\cF(\square_G)=\{F_I\mid I\in\mathcal{I}_G \text{ or } I\text{ is a singleton subset of }[\bar{n}]\}$. It follows from~\eqref{eq:Z_2-facet-vector} that the  $(\mathrm{mod}\,\,{2})$  characteristic function $\lambda\colon\cF(\square_G)\to \Z_2^n$ is given by
\begin{equation*}
    \lambda(F_I):=\begin{cases}
        \sum_{i\in I}\mathbf{e}_i&\text{ if }I\in\mathcal{I}_G,\\
        \mathbf{e}_i&\text{ if }I=\{\bar{i}\}  \text{ for some }i\in [n].
    \end{cases}
\end{equation*}  Let $\Lambda_G$ be the $\Z_2$-characteristic matrix of $\square_G$. Then $\Lambda_G$ is of size $n\times (|\mathcal{I}_G|+n)$.

\bigskip
\noindent\textbf{Simplicial complex $\mathbf{K_S^{\mathrm{odd}}}$ dual to $\boldsymbol{(\square_{G})_{{\omega}_{S}}}$.} \
Let   $S$ be a subset of $[n]$, and let $\omega_S$ be the sum of the $k$th rows of $\Lambda_G$ for all $k\in S$.
Then for each facet $F_I$, the $I$-entry of $\omega_S$ is nonzero if and only if either $I=\{\bar{i}\}$ for some $i\in S$ or $I\in \mathcal{I}_G$ such that $|I\cap S|$  is odd.
Hence $(\square_G)_{\omega_S}$ is the union of facets $F_I$ of $\square_G$, where the union is taken over all $I$ satisfying that
either $I=\{\bar{i}\}$ for some $i\in S$ or $I\in\mathcal{I}_G$ such that $|I\cap S|$ is odd.
We let $K^{\mathrm{odd}}_{G,S}$ be the dual of $(\square_G)_{\omega_S}$. Then $K^{\mathrm{odd}}_{G,S}$ is a simplicial complex since $\square_G$ is a simple polytope. We write  $K^{\mathrm{odd}}_G$ instead of $K^{\mathrm{odd}}_{G,[n]}$.
Note that the set of vertices\footnote{Note that a vertex is used for two meanings, one is for a graph and the other is for a simplicial complex. } of $K^{\mathrm{odd}}_{G,S}$ is
 $$\{ I\in \mathcal{I}_G \mid |I\cap S|\text{ is odd} \}  \cup  \{\bar{i}\mid i\in S\}. $$
Figure~\ref{ex:mainex2} shows examples for simplicial complexes $K^{\mathrm{odd}}_{G,S}$ and $K^{\mathrm{odd}}_{G[S]}$.
Since $(\square_G)_{\omega_S}$ and its dual $K_{G,S}^{\mathrm{odd}}$ have the same homotopy type, by Theorem~\ref{formula}, we have
\begin{eqnarray}\label{eq:common0}
       \beta^i (X^\R(\square_G))&=& \sum_{S\subseteq [n]} \tilde{\beta}^{i-1}(K^{\mathrm{odd}}_{G,S}).
\end{eqnarray}

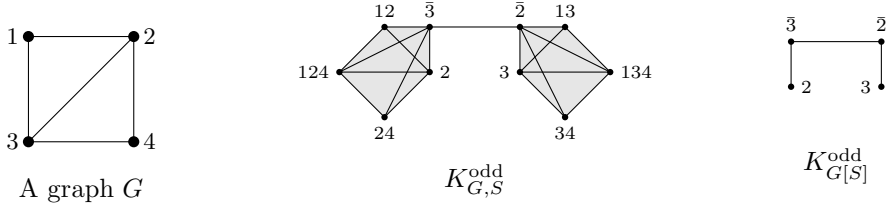
\begin{figure}[h]
    \begin{center}
 \begin{subfigure}{.3\textwidth}
        \centering
        \begin{tikzpicture}[scale=.7]
        \fill (0,3) circle (0pt);
        	\fill (0,0) circle(3pt);
        	\fill (2,0) circle(3pt);
        	\fill (0,2) circle(3pt);
        	\fill (2,2) circle(3pt);
        	\draw (0,0)--(0,2)--(2,2)--cycle;\draw (2,2)--(2,0)--(0,0);
        	\draw (-0.3,0) node{\footnotesize$3$};
        	\draw (2.3,0) node{\footnotesize$4$};
        	\draw (-0.3,2) node{\footnotesize$1$};
        	\draw (2.3,2) node{\footnotesize$2$};
        \end{tikzpicture}
        \caption*{A graph $G$}
    \end{subfigure}
    \begin{subfigure}{.3\textwidth}
    \centering
    \begin{tikzpicture}[scale=.6]
        \path (0,0) coordinate (2)
        (0,1) coordinate (-3)
        (2,1) coordinate (-2)
        (-2,0) coordinate (124)
          (-1,1) coordinate (12)
             (-1,-1) coordinate (24)
        (2,0) coordinate (3)    (4,0) coordinate (134)
        (3,1) coordinate (13)     (3,-1) coordinate (34);
                \fill[gray!20] (2)--(12)--(124)--(2)--cycle;
        \fill[gray!20] (2)--(24)--(124)--(12)--(-3)--(2)--cycle;
        \fill[gray!20] (3)--(34)--(134)--(13)--(-2)--(3)--cycle;
        \fill (2) circle (2pt) (12) circle (2pt)  (24) circle (2pt) (124) circle (2pt)
        (3) circle (2pt) (13) circle (2pt)  (34)circle (2pt) (134) circle (2pt);
        \draw (2) edge (12) edge (24) edge (124);
        \draw  (3) edge (13) edge (34) edge (134);
       \draw   (134) edge (13) edge (34);\draw   (124) edge (12) edge (24);
       \draw (-2) edge (-3) edge (3) edge (134) edge (13) edge (34);
       \draw (-3) edge (2) edge (12) edge (24) edge (124);
        \fill (-2) circle (2pt) (-3) circle (2pt);
        \fill (-2) node[above]{\tiny$\bar{2}$}
        (-3) node[above]{\tiny$\bar{3}$}
         (2) node[right]{\tiny$2$}
         (3) node[left]{\tiny$3$}
        (13) node[above]{\tiny$13$}
        (34) node[below]{\tiny$34$}
        (24) node[below]{\tiny$24$}
        (12) node[above]{\tiny$12$}
        (124) node[left]{\tiny$124$}
        (134) node[right]{\tiny$134$};
    \end{tikzpicture}
    \caption*{$K_{G,S}^{\mathrm{odd}}$}
    \end{subfigure}
    \begin{subfigure}{.25\textwidth}
\centering
    \begin{tikzpicture}[scale=.6]
        \path (0,0) coordinate (2)
        (0,1) coordinate (-3)
        (2,1) coordinate (-2)
        (-2,0) coordinate (124)
          (-1,1) coordinate (12)
             (-1,-1) coordinate (24)
        (2,0) coordinate (3)    (4,0) coordinate (134)
        (3,1) coordinate (13)     (3,-1) coordinate (34);
       \draw (-2) edge (-3) edge (3);
       \draw (-3) edge (2);
        \fill (-2) node[above]{\tiny$\bar{2}$}
        (-3) node[above]{\tiny$\bar{3}$}
        (2) node[right]{\tiny$2$}
        (3) node[left]{\tiny$3$};
        \fill (2) circle (2pt) (3)circle (2pt);
        \fill (-2) circle (2pt) (-3)circle (2pt);
    \end{tikzpicture}
    \caption*{$K^{\mathrm{odd}}_{G[S]}$}
    \end{subfigure}
    \end{center}
    \caption{The simplicial complexes $K_{G,S}^{\mathrm{odd}}$ and $K^{\mathrm{odd}}_{G[S]}$ when $S = \{2,3\}$.}\label{ex:mainex2}
    \end{figure}

Note that two simplicial complexes $K_{G,S}^{\mathrm{odd}}$ and $K^{\mathrm{odd}}_{G[S]}$ in Figure~\ref{ex:mainex2} are homotopy equivalent and they are contractible.
We will show that this phenomenon holds in general. More precisely, we will show that
  $K^{\mathrm{odd}}_{G,S}$ and $K^{\mathrm{odd}}_{G[S]}$ are homotopy equivalent  for any $S\subset [n]$ (Lemma~\ref{lem:reduced:subgraph}), and then  $K^{\mathrm{odd}}_{G[S]}$ is contractible when $G[S]$ is a connected even graph (Lemma~\ref{lem:even:contractible}).

We mention one useful lemma to prove Lemmata~\ref{lem:reduced:subgraph} and \ref{lem:even:contractible}.

\begin{lemma}[Lemma~5.2 of \cite{CP}] \label{lem:st}
        Let $I$ be a vertex of a simplicial complex  {$K$} and suppose that the link $\Lk_K I$ of $I$ in $K$ is contractible.
        Then {$K$} is homotopy equivalent to the complex ${K\setminus\St_K I}$, where $\St_K I$ is the star of $I$ in $K$.
\end{lemma}

We remark that for any two elements $I$ and $J$ in $K:=K^{\mathrm{odd}}_S$, it follows from~Proposition~\ref{prop:faces_meet} that
 $J\in \Lk_K I$ if and only if one of the following (a)$\sim$(d) is true: (a) $J\subsetneq I$, (b) $I\subseteq J$, (c) {$I\cup J\not\in\mathcal{I}_G$}, and (d) $J=\{\bar{j}\}$ for some $j\in [n]\setminus I$.

\begin{lemma}\label{lem:reduced:subgraph}
For any $S\subset [n]$, $K^{\mathrm{odd}}_{G,S}$ is homotopy equivalent to $K^{\mathrm{odd}}_{G[S]}$.
\end{lemma}
\begin{proof}For simplicity, we write  $K:=K_{G,S}^{\mathrm{odd}}$  and  $K':=K_{G[S]}^{\mathrm{odd}}$. Let $K^\ast$ be a minimal complex obtained by eliminating the stars of some vertices  in $K\setminus K'$ without changing the homotopy type. We will show that $K^\ast=K'$.

Suppose that $K^\ast\supsetneq K'$. Let us take a vertex $I$ in $K^\ast\setminus K'$ such that $|I\cap S|$ is minimal and $|I|$ is minimal.\footnote{We first check the minimality of $|I\cap S|$ and then check the minimality of $|I|$.}
Note that $I\in \mathcal{I}_G$ and $|I\cap S|$ is odd.
Then $G[I\cap S]$ has a connected component $I_1$ with an odd number of vertices. Clearly, $I_1\subset (I\cap S)$ and so ${I_1}\subset S$.
Thus, $I_1\in K'$ and $I_1\neq I$.
Note that if $J\in\Lk_{K^*} I$, then one of the following (a)$\sim$(d) is true : (a) $J\subsetneq I$, (b) $I\subseteq J$, (c) {$I\cup J\not\in\mathcal{I}_G$}, and (d) $J=\{\bar{j}\}$ for some $j\in [n]\setminus I$.
In the following, we will show that any element $J\in\Lk_{K^\ast} I$ also belongs to $\Lk_{K^*} I_1$, that is,  one of
 (a${}^\prime$)$\sim$(d${}^\prime$) is true: (a${}^\prime$) $J\subsetneq I_1$, (b${}^\prime$) $I_1\subseteq J$, (c${}^\prime$) {$I_1\cup J\not\in\mathcal{I}_G$}, and (d${}^\prime$) {$J=\{\bar{j}\}$ for some } $j\in[n]\setminus I_1$.

Suppose (a) $J\subsetneq I$. Then  $J\cap S$ is a subset of $I\cap S$, and hence $|J\cap S|\le |I\cap S|$ and $|J|<|I|$. Thus $J\in K'$ by the minimality conditions of $I$, that is, $J\subset S$.
Then $J$ is a subset of $S\cap I$ and so $J$ is a connected graph contained in a connected component of $G[I\cap S]$, which implies that either (a${}^\prime$) $J\subset I_1$ or (c${}^\prime$) {$I\cup J\not\in\mathcal{I}_G$}.

If (b) $I\subset J$, then (b${}^\prime$) $I_1\subset J$ holds.
If (c) {$I\cup J\not\in\mathcal{I}_G$}, then (c${}^\prime$) {$I_1\cup J\not\in\mathcal{I}_G$}.
If (d) $J=\{\bar{j}\}$ for some $j\in [n]\setminus I$, then  (d${}^\prime$) $j\in[n]\setminus I_1$ holds.

Therefore, $\Lk_{K^\ast} I$ is the cone with the vertex $I_1$, and so $\Lk_{K^\ast} I$ is contractible and $K^\ast$ is homotopy equivalent to $K^\ast\setminus \St_{K^\ast} I$ by Lemma~\ref{lem:st}. Then $K^\ast\setminus \St_{K^{\ast}} I$ is smaller than $K^\ast$, which contradicts the minimality of $K^\ast$. Therefore, $K^\ast=K'$.
\end{proof}

It follows from Lemma~\ref{lem:reduced:subgraph} that \eqref{eq:common0} is equivalent to the following:
\begin{eqnarray}\label{eq:common1}
       \beta^i (X^\R(\square_G))&=& \sum_{S\subseteq [n]} \tilde{\beta}^{i-1}(K^{\mathrm{odd}}_{G[S]}).
\end{eqnarray}

\begin{lemma}\label{lem:even:contractible}
If $G$ is a connected even graph, then
 $K_G^{\mathrm{odd}}$ is contractible.
\end{lemma}

\begin{proof}
For simplicity, let  $K:=K_G^{\mathrm{odd}}$.
Let $K'$ be the induced subcomplex of $K$ on the vertices $[\bar{n}]$, which is a simplex.
Let $K^\ast$ be a minimal complex obtained by eliminating the stars of some vertices $I\in \mathcal{I}_G$ such that $|I|$ is odd, without changing the homotopy type.

Suppose that $K^\ast\subsetneq K'$.
Take a vertex $I$  in $K^\ast\setminus K'$ such that $|I|$ is maximal.
Since $|I|$ is odd and $G$ is a connected even graph, there is a vertex $i\in [n]\setminus I$ of $G$ such that $I\cup\{i\}$ induces a connected graph.
Let $I_1=\{\bar{i}\}$.
Clearly, $I_1\in K'$.
We will show that any vertex in $\Lk_{K^*} I$ is in the link of $I_1$.
Take $J\in\Lk_{K^*} I$.
If $J=\{\bar{j}\}$ for some $j\in [n]\setminus I$, then $[n]\setminus I \subset [n]\setminus I_1$ and so $j\in [n]\setminus I_1$, which implies that $J$ is in the link of $I_1$.
Suppose that $J\in \mathcal{I}_G$ such that $|J|$ is odd. Then the maximality condition of $I$ implies that $I \not\subset J$.
If $J\subset I$, then $i\not\in J$.
If $G[I\cup J]$ is disconnected, then any neighbor of a vertex in $G[I]$ does not belong to $G[J]$ and so $i\not\in J$.
Therefore, $\Lk_{K^\ast} I$ is contractible and $K^\ast$ is homotopy equivalent to $K^\ast\setminus \St_{K^\ast} I$ by Lemma~\ref{lem:st}. Then $K^\ast\setminus \St_{K^\ast} I$ is smaller than $K^\ast$, which contradicts the minimality of $K^\ast$. Therefore, $K^\ast=K'$.
\end{proof}

\bigskip

Now we are ready to prove Theorem~\ref{thm:main:typeI}.
\begin{proof}[Proof of Theorem~\ref{thm:main:typeI}]
If $\kappa(G)={\kappa}$ and $G_1$,   $\ldots$, $G_{\kappa}$ are the connected components  of $G$, then
from definition,
 \[ K_G^{\mathrm{odd}} = K_{G_1}^{\mathrm{odd}}\ast \cdots \ast K_{G_\kappa}^{\mathrm{odd}} \simeq S^{\kappa-1}\wedge K_{G_1}^{\mathrm{odd}}\wedge\cdots\wedge K_{G_\kappa}^{\mathrm{odd}},\]
where $K_1*K_2$ denotes the simplicial join of $K_1$ and $K_2$.
Thus, \eqref{eq:common1} is equivalent to
\begin{eqnarray}\label{eq:common2}
    \beta^i (X^\R(\square_G)) &=& \sum_{S\subseteq [n]} \tilde{\beta}^{i-1}( K^{\mathrm{odd}}_{G[S]}) \quad = \quad \sum_{S\subseteq [n]} \left( \sum_{\sum k_j=i-\kappa(G[S])}   \prod_{j} \tilde{\beta}^{k_j}    (K^{\mathrm{odd}}_{G[S_j]}) \right),
\end{eqnarray}
where $G[S_j]$ means the $j$th component of $G[S]$.
If $G[S]$ is not odd, then $G[S]$ has a connected component with an even number of vertices, and then by Lemma~\ref{lem:even:contractible}, $\tilde{\beta}^{i-1}( K^{\mathrm{odd}}_{G[S]})=0$.
Thus \eqref{eq:common2}  is equivalent to
\begin{eqnarray*}\label{eq:common3}
    \beta^i (X^\R(\square_G)) &=& \sum_{S\subseteq [n]\atop G[S]\text{ is odd}} \tilde{\beta}^{i-1}( K^{\mathrm{odd}}_{G[S]}).
\end{eqnarray*}
For a graph $H$, let us denote by $(\partial \square_H)^*$ the simplicial sphere which is the dual complex of $\partial \square_H$. Then $K^{\mathrm{odd}}_{H}$ is an induced subcomplex of $(\partial \square_H)^*$. We denote by $K^{\mathrm{even}}_{H}$ the induced subcomplex of $(\partial \square_H)^*$ on the vertices not belonging to $K^{\mathrm{odd}}_{H}$. Note that the vertices of $K^{\mathrm{even}}_{H}$ bijectively correspond to the connected even induced subgraphs of $H$. {Since $(\partial \square_{G[S]})^*$ is a sphere of dimension $|S|-1$}, the Alexander duality implies
\begin{eqnarray}\label{eq:common4}
       \beta^i (X^\R(\square_G))&=& \sum_{S\subseteq [n]\atop G[S]\text{ is odd}} \tilde{\beta}_{|S|-i-1}(K^{\mathrm{even}}_{G[S]}).
\end{eqnarray}

When $|S|=1$, $K^{\mathrm{even}}_{G[S]}$ is an empty simplicial complex. In this case we regard it as a sphere of dimension $-1$ for the formula above.

For any odd graph $H$, $K^{\mathrm{even}}_{H}$ is homotopy equivalent to the order complex of the poset
$\mathcal{P}_H^{\mathrm{even}}\setminus \{\hat{0},\hat{1}\}$ (see Lemma~4.7 of \cite{CP}). The poset $\mathcal{P}_{H}^{\mathrm{even}} \setminus \{\hat{0},\hat{1}\}$ is pure and shellable by Theorem~\ref{prop:simple-shellable}, and its length is $\frac{|S|-\kappa(G[S])}{2} - 1$.
Hence,
\[ \tilde{\beta}_{|S|-i-1}(K^{\mathrm{even}}_{G[S]}) = \left|\mu(\mathcal{P}^{\mathrm{even}}_{G[S]})\right| = \begin{cases}
  \left|\sum_{H\sqsubseteq G[S]} a(H)\right| &\text{if }|S|-i-1=\frac{|S|-\kappa(G[S])}{2}-1;\\
  0 &\text{otherwise.}
\end{cases} \]
Thus, together with~\eqref{eqn:bwrittenbya}, \eqref{eq:common4} is equivalent to the following
\begin{eqnarray}\label{eq:common5}
       \beta^i (X^\R(\square_G))&=& \sum_{S\subseteq [n]\atop G[S]\text{ is odd}} \left|b(G[S])\right| ,
\end{eqnarray}
where $|S|-i=\frac{|S|-\kappa(G[S])}{2}$, that is, $|S|+\kappa(G[S])=2i$.
Note that \eqref{eq:common5} is equivalent to
    $$
    \beta^i(X^\R(\square_G))= \sum_{H \sqsubseteq G\atop |V(H)|+\kappa(H)=2i}\left|b(H)\right|,
    $$
which completes the proof of Theorem~\ref{thm:main:typeI}.

\end{proof}

\begin{remark}
It follows from Theorem~\ref{prop:simple-shellable} and Lemmas~\ref{lem:reduced:subgraph} and~\ref{lem:even:contractible} that $(\square_G)_{\omega_S}$ (respectively, $(\triangle_G)_{\omega_S}$) is torsion-free for every $S\subseteq [n]$, and hence $H^\ast(X^\R(\square_G))$ (respectively, $H^\ast(X^\R(\triangle_G))$)  is completely determined by the $b$-numbers (respectively, $a$-numbers) of all $H\sqsubset G$ and the $h$-vector of $\square_G$ (respectively, $\triangle_G$). Furthermore, for an octopus $G$, the real toric manifolds $X^\R(\triangle_G)$ and $X^\R(\square_{L(G)})$ have the same cohomology groups.
\end{remark}

\section{Examples}\label{sec:example}
In this section, we provide {some interesting integer sequences arising from the $b$-number of a graph $G$ and the Betti numbers of $X^\R(\square_G)$ for some graph families} such as paths, cycles, complete graphs, and stars.

\begin{corollary}[\cite{CP}]\label{cor:b_numbers}
For an odd integer $n$, we have
$$
b(G)=\begin{cases}
              (-1)^{\frac{n+1}{2}}A_{n} & \text{if }G=K_{n},\\
               (-1)^{\frac{n+1}{2}}\mathrm{Cat(\frac{n-1}{2})} &  \text{if }G=P_{n},\\
                (-1)^{\frac{n-1}{2}}\binom{n-1}{(n-1)/2}&\text{if }G=C_{n},\\
                 (-1)^{\frac{n+1}{2}}A_{n-1}&\text{if }G=K_{1,n-1},\\
            \end{cases}
        $$
 where $A_{k}$ is the $k$th Euler Zigzag number and
$\mathrm{Cat}(n)$ is the $n$th Catalan number.
\end{corollary}

For paths and complete graphs, their formulae for the Betti numbers are directly obtained by Corollary~\ref{cor:number} and Proposition~\ref{prop:spider:octopus}.
From the equivalence of the normal fans of $\square_{P_n}$ and $\triangle_{P_{n+1}}$ in Proposition~\ref{prop:spider:octopus}, a corollary follows from Corollary~1.4 of \cite{CP}. See the table on the left side of Table~\ref{tabl:cycletable}, and it
makes up the Catalan's triangle.
\begin{corollary}\label{cor:bofpath}
For any integer $i\ge 0$, we have \[
        \beta^i(X^\R(\square_{P_n})) = \begin{cases}
          \binom{n+1}{i} - \binom{n+1}{i-1}&\text{ if }1 \le i \le \lfloor\frac {n+1}2\rfloor,\\
          0 & \text{ otherwise}.\\
        \end{cases}
    \]
\end{corollary}

For a complete graph $K_{n}$ with $n$ vertices and a star $K_{1,n}$ with $n$ leaves, recall that the normal fan of $\square_{K_n}$ is equivalent to that of $\triangle_{K_{1,n}}$ in Proposition~\ref{prop:spider:octopus}. Hence from Corollary~\ref{cor:number}, we have the following.
\begin{corollary}[Corollary~1.6 of \cite{CP}]
  For any integer $i\ge0$, we have
    \[
        \beta^i(X^\R(\square_{K_{n}})) = \binom{n}{2i-1}A_{2i-1}.
    \]
\end{corollary}

\begin{table}[b!]
\begin{minipage}{0.45\textwidth}
\centering\begin{tabular}{|c|cccccc|}
    \hline
        \backslashbox{$n$}{$i$}& $0$ & 1 & 2 & 3 & 4 & 5  \\ \hline
    $1$ & 1 & 1 &  &  &  &    \\
    2 & 1 & 2 &  &  &  &    \\
    3 & 1 & 3 & 2 &  &  &    \\
    4 & 1 & 4 & 5 &  &  &    \\
    5 & 1 & 5 & 9 & 5 &  &    \\
    6 & 1 & 6 & 14 & 14 &  &     \\
    7 & 1 & 7 & 20 & 28 & 14 &    \\
    8 & 1 & 8 & 27 & 48 & 42 &     \\
    9 & 1 & 9 & 35 & 75 & 90 & 42    \\
    \hline
\end{tabular}\caption*{$\beta^i(X^\R(\square_{P_n}))$}
\medskip
\end{minipage}\hfill
\begin{minipage}{0.45\textwidth}
\centering
\begin{tabular}{|c|cccccc|}
    \hline
    \backslashbox{$n$}{$i$}& $0$ & 1 & 2 & 3 & 4 & 5  \\ \hline
    $1$ & 1 & 1& & & &   \\
    $2$ & 1 & 2 &  &  &  &    \\
    3 & 1 & 3 & 2 &  &  &    \\
    4 & 1 & 4 & 6 &  &  &    \\
    5 & 1 & 5 & 10 & 6 &  &    \\
    6 & 1 & 6 & 15 & 20 &  &    \\
    7 & 1 & 7 & 21 & 35 & 20 &     \\
    8 & 1 & 8 & 28 & 56 & 70 &    \\
    9 & 1 & 9 & 36 & 84 & 126 &  70   \\
    \hline
\end{tabular}\caption*{$\beta^i(X^\R(\square_{C_n}))$}
\medskip
\end{minipage}\caption{The Betti numbers of $X^\R(\square_{P_n})$ and $X^\R(\square_{C_n})$.}\label{tabl:cycletable}
\end{table}

The graph cubeahedron corresponding to a cycle is called a halohedron in \cite{DHV}.
We are going to compute the Betti numbers of $X^\R(\square_{C_n})$. See the table on the right side of Table~\ref{tabl:cycletable}.

By a \emph{word} we mean a finite sequence consisting of given alphabets. Recall that a \emph{Dyck word} of length $2k$ is a grammatically correct expression consisting of $k$ left parentheses `(' and $k$ right parentheses `)'. It is well-known that the number of Dyck words of length $2k$ is the $k$th Catalan number $\Cat(k)$.

Let $\pi\colon \Z \to \Z_n$ be the canonical quotient map. We say that a  map  $f\colon \Z_n \to \{(,),*\}$ is a \emph{partial Dyck word} on $\Z_n$ if there are finitely many intervals of integers $N_1,\dotsc,N_k \subset \Z$ such that
\begin{enumerate}
  \item $\{ \pi(N_1),\dotsc,\pi(N_k)\} $ is a partition of $\Z_n$ and
  \item $f \circ \pi$ restricted to each of $N_1,\dotsc,N_k$ induces either a Dyck word or a word $*$.
\end{enumerate}
For example,
\[()()***(())*, \qquad )**((())()(), \qquad **((())()())\]
 are partial Dyck words and $(()*())*((*))*$ is not. We also note that the second one and third one are distinguished  as partial Dyck words, even though they are the same up to rotation.
We say a parenthesis is in inside of the other parenthesis if they are in a  same interval in the restriction of $f$ and one is inside of the others in the interval.
 Some of the parentheses are \emph{outermost}, that is, they are not inside of other parentheses. For example, the shaded ones in the following are outermost parentheses.
\[\colorbox{gray!10}{\!\textbf{()()}\!\!} *** \colorbox{gray!10}{\!\textbf{(}\!\!}()\colorbox{gray!10}{\!\!\textbf{)}\!}*, \qquad
\colorbox{gray!10}{\!\textbf{)}\!\!}**\colorbox{gray!10}{\!\textbf{(}\!\!}(())()(), \qquad
**\colorbox{gray!10}{\!\textbf{(}\!\!}(())()()\colorbox{gray!10}{\!\!\textbf{)}\!\!}\]
Parentheses which are not outermost are called \emph{inner}.

\begin{theorem}\label{cor:bofcycle}
For any integer $i\ge0$, we have
 \[
        \beta^i(X^\R(\square_{C_n})) =\begin{cases}
               \binom{n}{i}     & \hbox{if $1 \le i \le \lfloor \frac{n}{2} \rfloor$,}\\
                        \binom{n-1}{i-1}   & \hbox{if $n$ is odd and $i = \frac{n+1}2$,} \\
                       0  & \hbox{otherwise.}
        \end{cases}
    \]
\end{theorem}
\begin{proof}
We apply Theorem~\ref{thm:main:typeI} to a cycle  $C_n$.
If $i > n/2$, then the only possible nontrivial case is when $H = C_n$ is the whole graph and $i=(n+1)/2$, and  $\beta^{i}(X^\R(\square_{C_n})) = |b(H)| = \binom{n-1}{(n-1)/2}$ by Corollary~\ref{cor:b_numbers}.

Now we suppose that $0 \le i \le n/2$.
We identify $V(C_n)$ with $\Z_n$ so that $\{j,j+1\}$ is an edge of $C_n$ for $j\in \Z_n$.
For a partial Dyck word $f:\Z_n \rightarrow \{(,),*\}$, consider the set
\[I_f=\{ j\in \Z_n \mid f(j)\text{ is a parenthesis which is either inner or left outermost}\}.\]
In other words, $I_f$ excludes $*$ and right outermost parentheses.
For simplicity, for each partial Dyck word $f$, let $H_f=C_n[I_f]$. Then $H_f$ is odd and $|V(H_f)| + \kappa(H_f)$ is equal to the number of parentheses in $f$, that is, $2i$.

Let $\mathcal{D}_i$ be the set of all partial Dyck words having exactly $i$ left parentheses.
We define an equivalence relation $\sim$ on $\mathcal{D}_i$ by $f\sim g$ if and only if the inverse image of the outermost parentheses in $f$ is equal to that in $g$.
If a partial Dyck word $f$ has $q$ pairs of outermost parentheses and the $j$th pair has $2k_j$ inner parentheses, then
the size of the equivalence class $[f]$ is
\[  |[f]|=\Cat(k_1)\times \cdots \times \Cat(k_q)= \prod_{j}|b(P_{2k_j+1})|=|b(H_f)|,\]
where the second equality is from $\Cat(k_j)=|b(P_{2k_j+1})|$ by Corollary~\ref{cor:b_numbers}, and the last equality is from the fact that the $j$th component of $H_f$ is a path with $2k_j+1$ vertices  and from the definition of the $b$-number.

 On the other hand,  for any odd induced subgraph $H$ of $C_n$ such that $|V(H)| + \kappa(H)=2i$, there is a partial Dyck word $f\in \mathcal{D}_i$ such that $H_f=H$.
Thus, together with Theorem~\ref{thm:main:typeI},
\[\beta^i(X^\R(\square_{C_n})) = \sum_{H \sqsubseteq G\atop |V(H)|+\kappa(H)=2i} |b(H)| \quad =  \quad \sum_{[f]\in \mathcal{D}_i/\sim} |b(H_f)|\quad = \quad \sum_{[f]\in \mathcal{D}_i/\sim} |[f]|\quad = \quad |\mathcal{D}_i|, \]
where $\mathcal{D}_i/\sim$ is the set of all equivalence classes.

It remains to show that $|\mathcal{D}_i|= \binom{n}{i}$ for any integers $i$ and $n$ with $2i\le n$.
For given $f\in\mathcal{D}_i$, one takes the set $\{i \in \Z_n \mid f(i) = (\}$, and this set is distinguishable by $f$, and thus this gives an injective function from $\mathcal{D}_i$ to the set of all $i$-subsets of $\Z_n$.
To show that this is surjective, take a subset $I \subseteq \Z_n$ such that $|I|=i$.
Since $i\le \frac{n}{2}$, there exists $j_1 \in I$ such that $j_1+1 \notin I$. Then one assigns $f(j_1) = ($ and $f(j_1+1) =~)$, and then removes both ones to get a subset $I'=I\setminus\{j_1\}$ of $\Z_{n-2}$. Again, find $j_2 \in I'$ such that $j_2+1 \notin I'$ and then assign $f(j_2) = ($ and $f(j_2+1) =~)$. In this way, we can assign $i$ )'s inductively, and then we assign $*$ for remaining $n-2i$ elements.
\end{proof}

For a star $K_{1,n}$, we have the following result.
\begin{proposition}\label{prop:bofstar} For any integer $i\ge 0$, we have
    \[
        \beta^i(X^\R(\square_{K_{1,n}})) = \binom{{n}}{i} + \binom{{n}}{2i-2}A_{2i-2},
    \] where $A_{-2}=0$.
\end{proposition}
\begin{proof}
    Note that by Corollary~\ref{cor:b_numbers}, $b(K_{1,2k}) = (-1)^{k+1}A_{2k}$.
    Each odd induced subgraph of $K_{1,n}$ is an edgeless graph induced from the leaves or a star $K_{1,2k}$ for $0\leq k\leq n/2$. Hence an odd induced subgraph $H$ such that $|V(H)|+\kappa(H)=2i$ is either the edgeless graph with $i$ isolated vertices or  a subgraph isomorphic to $K_{1,2i-2}$. Hence, the proof is done.
\end{proof}

\begin{remark}
    For a graph $G$ with $n$ vertices, one has $\beta^i(X^\R(\triangle_G)) = 0$ if $i > n/2$.
    For the graph cubeahedron $\square_G$,
    it can happen that $\beta^i(X^\R(\square_G)) > 0$ even though $i > \frac{n+1}2$ as in Proposition~\ref{prop:bofstar}.
\end{remark}

\section{Remarks}\label{sec:remark}
    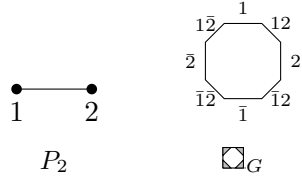
\begin{figure}[t]
    \centering
    \begin{subfigure}[b]{.14\textwidth}
    \centering
        \begin{tikzpicture}[scale=1]
            \fill (0,0) circle(2pt);
            \fill (1,0) circle(2pt);
            \draw (0,0)--(1,0);
            \draw (0,-0.3) node{$1$};
            \draw (1,-0.3) node{$2$};
        \end{tikzpicture}
        \caption*{$P_2$}
    \end{subfigure}
    \begin{subfigure}[b]{.14\textwidth}
    \centering
        \begin{tikzpicture}[scale=1]
            \draw (0,0.25)--(0.25,0)--(0.75,0)--(1,0.25)--(1,0.75)--(0.75,1)--(0.25,1)--(0,0.75)--cycle;
            \draw (0.5,-0.2) node{\tiny$\bar1$};
            \draw (0.5,1.2) node{\tiny$1$};
            \draw (-0.2,0.5) node{\tiny$\bar2$};
            \draw (1.2,0.5) node{\tiny$2$};
            \draw (1,1) node{\tiny$12$};
            \draw (0,0) node{\tiny$\bar1 \bar2$};
            \draw (1,0) node{\tiny$\bar1 2$};
            \draw (0,1) node{\tiny$1\bar2$};
        \end{tikzpicture}
        \caption*{$\CIII_G$}
    \end{subfigure}\caption{A graph associahedron of type B}\label{fig:typeB}
    \end{figure}

 A graph cubeahedron is obtained from the standard cube $\square^n$ by truncating the faces labeled by $I$ for each $I\in\mathcal{I}_G$
 in increasing order of dimensions. We introduce a slightly different polytope, which is also made from $\square^n$ by truncating the faces.
Given a graph $G$, we denote by  $\CIII_G$  a polytope by truncating the faces labeled $\widetilde{I}$ (simultaneously)
 for each $I\in\mathcal{I}_G$
 in increasing order of dimensions, where $i$ or $\overline{i}$ belongs to $\widetilde{I}$ if and only if $i\in I$.
Note that for the complete graph $K_n$, $\CIII_{K_n}$ is known as the type $B$ permutohedron.
So we call this polytope $\CIII_G$ a \emph{graph associahedron of type $B$}. Note that $\CIII_G$ is also  a Delzant polytope.
For example,
    consider a complete graph $K_2$ with two vertices, and then the following Figure~\ref{fig:typeB} shows a graph associahedron of type $B$.

The authors already checked that all statements of this paper corresponding to $\CIII_G$ are well-established except shellability. More precisely, we can apply Theorem~\ref{formula} and Lemma~\ref{lem:reduced:subgraph} to the case of $\CIII_G$. Then the problem of finding the Betti numbers of a real toric manifold associated with $\CIII_G$ is  converted to studying a topology of the order complex of some special poset $\mathcal{Q}_G^{\mathrm{even}}$, which is  defined by
$$ \mathcal{Q}_G^{\mathrm{even}} = \{ I\subset [n]\cup[\bar{n}] : I^+\cap I^-=\emptyset,\, G[I^+\cup I^-] \text{ is even}\} \cup \{\hat{0},\hat{1}\},$$
where
\[I^+ =I\cap[n] \quad \text{ and }\quad
I^-=\{ i\in[n] \mid \bar{i} \in I \}.\]
For a complete graph $G$, the Type $B$ permutohedron  $\CIII_G$ is already studied in \cite{CPP_B} and it was shown that $\mathcal{Q}_G^{\mathrm{even}}$ is shellable.
    \begin{figure}[t]
    \centering
    \begin{subfigure}[b]{.2\textwidth}
    \centering
        \begin{tikzpicture}[scale=1]
            \fill (0,0) circle(2pt);
            \fill (1,0) circle(2pt);
            \fill (-1,0) circle(2pt);
          \fill (0,1) circle(2pt);          \fill (0,-1) circle(2pt);
            \draw (0,0)--(1,0);    \draw (0,0)--(-1,0);
            \draw (0,0)--(0,1);      \draw (0,0)--(0,-1);
            \draw (0.1,-0.2) node{\tiny$1$};
           \draw (-1,-0.3) node{\tiny$4$};
           \draw (0,1.2) node{\tiny$3$};
             \draw (1,-0.3) node{\tiny$2$};        \draw (0.2,-1) node{\tiny$5$};
                      \draw (0,-2) node{\small$G=K_{1,4}$};
        \end{tikzpicture}
    \end{subfigure}
    \begin{subfigure}[b]{.6\textwidth}
    \centering
    	\begin{tikzpicture}[scale=.6]
            \node [draw] (0) at (0,-1) {\tiny $\emptyset$};
       \node [draw] (1234) at (-7.5,4) {\!\!\tiny$1234$\!\!};
       \node [draw] (1b234) at (-6,4) {\!\!\tiny$1\bar{2}34$\!\!};
       \node [draw] (12b34) at (-4.5,4) {\!\!\tiny$12\bar{3}4$\!\!};
       \fill (-3,4) circle(0.8pt) (-2.6,4) circle(0.8pt)(-2.2,4) circle(0.8pt);
       \node [draw] (1b2b3b4) at (-1,4) {\!\!\tiny$1\bar{3}\bar{4}\bar{5}$\!\!};
              \node [draw] (b1234) at (1,4) {\!\!\tiny$\bar{1}234$\!\!};
       \node [draw] (b1b234) at (2.5,4) {\!\!\tiny$\bar{1}\bar{2}34$\!\!};
       \node [draw] (b12b34) at (4,4) {\!\!\tiny$\bar{1}2\bar{3}4$\!\!};
       \fill (5.5,4) circle(0.8pt) (5.9,4) circle(0.8pt)(6.3,4) circle(0.8pt);
       \node [draw] (b1b2b3b4) at (7.5,4) {\!\!\tiny$\bar{1}\bar{3}\bar{4}\bar{5}$\!\!};
 \node [draw] (12) at (-8,2) {\!\!\tiny$12$\!\!};
        	\node [draw] (13) at (-7,2) {\!\!\tiny$13$\!\!};
        	\node [draw] (14) at (-6,2) {\!\!\tiny$14$\!\!};
        	\node [draw] (15) at (-5,2) {\!\!\tiny$15$\!\!};
            \node [draw] (1b2) at (-4,2) {\!\!\tiny$1\bar{2}$\!\!};
        	\node [draw] (1b3) at (-3,2) {\!\!\tiny$1\bar{3}$\!\!};
        	\node [draw] (1b4) at (-2,2) {\!\!\tiny$1\bar{4}$\!\!};
        	\node [draw] (1b5) at (-1,2) {\!\!\tiny$1\bar{5}$\!\!};
            \node [draw] (b12) at (1,2) {\!\!\tiny$\bar{1}2$\!\!};
        	\node [draw] (b13) at (2,2) {\!\!\tiny$\bar{1}3$\!\!};
        	\node [draw] (b14) at (3,2) {\!\!\tiny$\bar{1}4$\!\!};
        	\node [draw] (b15) at (4,2) {\!\!\tiny$\bar{1}5$\!\!};
          \node [draw] (b1b2) at (5,2) {\!\!\tiny$\bar{1}\bar{2}$\!\!};
        	\node [draw] (b1b3) at (6,2) {\!\!\tiny$\bar{1}\bar{3}$\!\!};
        	\node [draw] (b1b4) at (7,2) {\!\!\tiny$\bar{1}\bar{4}$\!\!};
            	\node [draw] (b1b5) at (8,2) {\!\!\tiny$\bar{1}\bar{5}$\!\!};
            \node [draw] (G) at (0,7) {\tiny $[n]\cup[\bar{n}]$};
  \path (0) edge (13) edge (12) edge (14)  edge (15) edge (b13) edge (b12) edge (b14) edge (b15) edge (1b3) edge (1b2) edge (1b4) edge (1b5) edge (b1b3) edge (b1b2) edge (b1b4) edge (b1b5);
  \path[dotted,thick] (0,0) edge (0,6);
   \path (G) edge (1234) edge (1b234) edge (12b34) edge (1b2b3b4) edge (b1234) edge (b1b234) edge (b12b34) edge (b1b2b3b4);	
   \path (1234) edge (12) edge (13) edge (14);
   \path (1b234) edge (1b2) edge (13) edge (14);
   \path (12b34) edge (12) edge (1b3) edge (14);
   \path (1b2b3b4) edge (1b3) edge (1b4) edge (1b5);
      \path (b1234) edge (b12) edge (b13) edge (b14);
   \path (b1b234) edge (b1b2) edge (b13) edge (b14);
   \path (b12b34) edge (b12) edge (b1b3) edge (b14);
   \path (b1b2b3b4) edge (b1b3) edge (b1b4) edge (b1b5);
    	\end{tikzpicture}
        \caption*{$\mathcal{Q}_G^{\mathrm{even}}$}
    \end{subfigure}\caption{When $G=K_{1,4}$, $\mathcal{Q}_G^{\mathrm{even}}$ is not shellable. }\label{fig:remark}
    \end{figure}
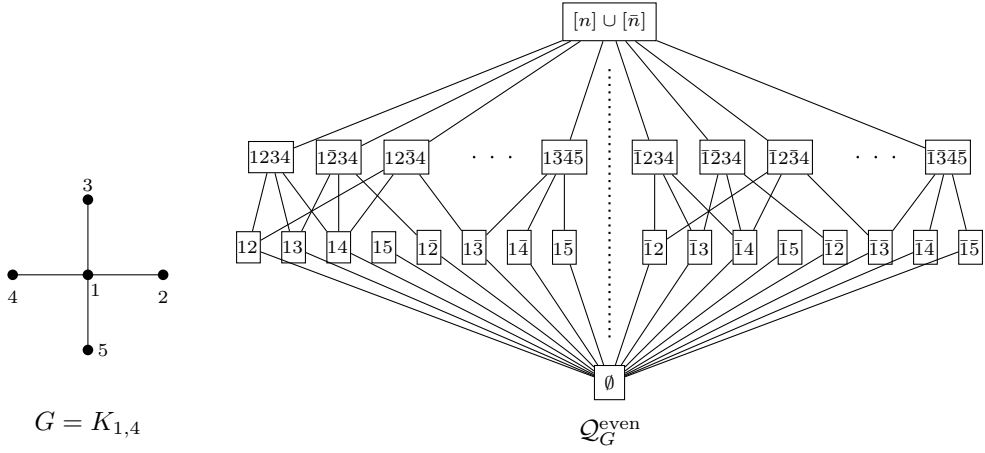
However, $\mathcal{Q}_G^{\mathrm{even}}$ is not always shellable, and so it is not easy to compute the Betti numbers of the real toric manifold associated with $\CIII_G$. Consider a star $K_{1,n}$  $(n\ge 4)$ as in Figure~\ref{fig:remark}. Then for any two elements $I$ and $J$ of $\mathcal{Q}_G^{\mathrm{even}}$, if $1\in I$ and $\bar{1}\in J$, then any maximal chain containing $I$ and any maximal chain containing $J$ do not intersect. Thus $\mathcal{Q}_G^{\mathrm{even}}$ cannot be shellable.

\section*{Acknowledgement}
    The authors thank to Prof. Vincent Pilaud and Dr. Thibault Manneville for their kind answer to the questions for the compatibility fans. The authors also thank to Yusuke Suyama for his valuable comment.

\end{document}